\documentclass[11pt]{amsart}
\usepackage{hyperref}
\usepackage{amsfonts,mathrsfs,bbm,latexsym,rawfonts,amsmath,amssymb,amsthm,epsfig}
\usepackage{fullpage, setspace, color}
\allowdisplaybreaks

\newtheorem{theorem}{Theorem}[section]

\newtheorem{thm}{Theorem}[section]

\newtheorem{rem}[thm]{Remark}
\newtheorem{lem}[thm]{Lemma}
\newtheorem{prop}[thm]{Proposition}
\newtheorem{defn}{Definition}[section]

\numberwithin{equation}{section}

\newcommand{\al}{\alpha}

\newcommand{\ld}{\lambda}

\newcommand{\de}{\delta}
\newcommand{\De}{\Delta}
\newcommand{\ep}{\varepsilon}

\newcommand{\om}{\omega}
\newcommand{\Om}{\Omega}
\newcommand{\ga}{\gamma}

\renewcommand{\th}{\theta}

\renewcommand{\P}{\mathcal{P}}

\newcommand{\g}{\mathfrak{g}}

\newcommand{\J}{\mathfrak{J}}

\renewcommand{\S}{\mathbb{S}}

\newcommand{\Real}{\mathbb{R}}

\newcommand{\norm}[1]{\Vert#1\Vert}

\def\<{\left\langle} \def\>{\right\rangle}
\def\({\left(} \def\){\right)}
\newcommand{\n}{\nabla}
\newcommand{\p}{\partial}

\begin{document}
\title{very regular solution to Landau-Lifshitz system with spin-polarized transport }
\author{Bo Chen}
\address{1. Department of Mathematical Sciences, Tsinghua University, Beijing, 100084, China; 2. Academy of Mathematics and systems science, Chinese Academy of Sciences, Beijing 100190, China; 3. School of Mathematical Sciences, University of Chinese Academy of Sciences, Beijing 100049, China.}
\email{chenbo@mail.tsinghua.edu.cn}
\author{Youde  Wang}
\address{1. School of Mathematics and Information Sciences, Guangzhou University; 2. Hua Loo-Keng Key Laboratory of Mathematics, Institute of Mathematics, Academy of Mathematics and Systems Science, Chinese Academy of Sciences, Beijing 100190, China; 3. School of Mathematical Sciences, University of Chinese Academy of Sciences, Beijing 100049, China.}
\email{wyd@math.ac.cn}
 \keywords{Landau-Lifshitz system with spin-polarized transport, Very regular solution, The compatibility conditions of the initial data, Galerkin approximation method, Auxiliary approximation equation}
%\date{\today}
\begin{abstract}
In this paper, we provide a precise description of the compatibility conditions for the initial data so that one can show the existence and uniqueness of regular short-time solution to the Neumann initial-boundary problem of a class of Landau-Lifshitz-Gilbert system with spin-polarized transport, which is a strong nonlinear coupled parabolic system with non-local energy.
\end{abstract}
\maketitle
%\tableofcontents
%\newpage

\section{Introduction}\label{s: intro}
\subsection{Background}
In physics, the Landau-Lifshtiz (LL) equation is a fundamental evolution equation for the ferromagnetic spin chain and was proposed on the phenomenological background in studying the dispersive theory of magnetization of ferromagnets. It was first deduced by Landau and Lifshitz in \cite{LL}, and then proposed by Gilbert in \cite{G} with dissipation. In fact, this equation describes the Hamiltonian dynamics corresponding to the Landau-Lifshitz energy, which is defined as follows.

Let $\Omega$ be a smooth bounded domain in the Euclidean space $\mathbb{R}^3$, whose coordinates are denoted by $x=(x^1,x^2,x^3)$. We assume that a ferromagnetic material occupies the domain $\Omega\subset\mathbb{R}^3$. Let $u$, denoting magnetization vector, be a mapping from $\Omega$ into a unit sphere $\S^2\subset\mathbb{R}^3$. The Landau-Lifshitz energy of map $u$ is defined by
$$\mathcal{E}(u):=\int_{\Om}\Phi(u)\,dx+\frac{1}{2}\int_{\Om}|\n u|^2\,dx-\frac{1}{2}\int_{\Omega}h_d\cdot u\,dx.$$
Here $\n $ denote the gradient operator and $dx$ is the volume element of $\mathbb{R}^3$.

In the above Landau-Lifshitz functional, the first and second terms are the anisotropy and exchange energies, respectively. $\Phi(u)$ is a real function on $\S^2$. The last term is the self-induced energy, and $h_d = -\nabla w$ is the demagnetizing field. In fact, the magnetostatic potential $w$ has precise formula 
$$w(x) = \int_{\Omega} \nabla N(x-y)u(y)dy,$$
where $N(x) = -\frac{1}{4\pi|x|}$ is the Newtonian potential in $\mathbb{R}^3$.

The LL equation with dissipation can be written as
\[
u_t - \alpha u\times u_t =-u\times h,\]
where ``$\times$" denotes the cross production in $\Real^{3}$ and the local field $h$ of $\mathcal{E}(u)$ can be derived as
$$h:=-\frac{\delta\mathcal{E}(u)}{\delta u}= \Delta u + h_d -\nabla_u\Phi.$$
Here, the constant $\alpha$ is the damping parameter, which is characteristic of the material, and is usually called the Gilbert damping coefficient.  Hence the Landau-Lifshitz equation with damping term is also called the Landau-Lifshitz-Gilbert (LLG) equation in the literature.

On the other hand, another new physical model for the spin-magnetization system, which takes into account the diffusion process of the spin accumulation through the multilayer, has attracted considerable attention of many mathematicians. Especially,  it was originally presented by Zhang et al. \cite{ZF, SLZ}, and later be extended by \cite{GW} to three dimensions of the model for spin-polarized transport. In this paper, we call the model as the Landau-Lifschitz-Gilbert equation with spin-polarized transport (LLGSP), which is given in below
\begin{equation}\label{llgsp}
\begin{cases}
\p_t u-\al u\times \p_t u=-u\times(h+s),\quad\quad&\text{(x,t)}\in\Om\times \Real^+,\\[1ex]
\p_t s=-\mbox{div} J_s-D_0(x)\cdot s-D_0(x)\cdot s\times u, \quad\quad&\text{(x,t)}\in\Om_0\times \Real^+,
\end{cases}
\end{equation}
with initial-boundary condition
\begin{equation}\label{in-bd}
\begin{cases}
u(\cdot,0)=u_0:\Om\to \S^{2},\,\,\frac{\p u}{\p \nu}|_{\p \Om}=0,\\[1ex]
s(\cdot,0)=s_0:\Om_0\to \Real^3,\,\,\frac{\p s}{\p \nu}|_{\p \Om_0}=0.
\end{cases}
\end{equation}

Here, $\Om_0\subset\Real^3$ is a bounded domain and $\Om\subset\Om_0$, $u(x,t)\in\S^2$ is the magnetized field, $s(x,t)\in \Real^{3}$ is the spin accumulation, $D_0(x)$ is  a positive measurable function, which represents the diffusion coefficient of the spin accumulation, $\al>0$ is the Gilbert damping parameter, and $J_s$ is the spin current given by
$$J_s=u\otimes J_e-D_0\cdot\{\n s-\th u\otimes(\n s\cdot u)\}=u\otimes J_e+A(u)\n s,$$
where $J_e$ is the applied electric current, $\th\in(0,1)$ is the spin polarization parameter and the coefficient matrix $A(u)$ is expressed as 
\[A(u)=-D_0
\begin{pmatrix}
	1-\th u^2_1&-\th u_1u_2&-\th u_1u_3\\
	-\th u_2u_1&1-\th u^2_2&-\th u_2u_3\\
	-\th u_3u_1&-\th u_3u_2&1-\th u^2_3\\
\end{pmatrix}
.\]
 The spin accumulation $s$ is defined on $\Om_0$ and the magnetization $u$ is defined on the magnetic domain $\Om$ and extended as zero outside.  The additional term in the LLG equation is induced by the interaction $F[u,s]=-\int_{\Om}u\cdot s\,dx$.

However, we are concerned in this paper with the existence of regular solution to system \eqref{llgsp}, so it is natural to assume that $\Om=\Om_0$ and the boundary of $\Om$ is smooth. For simplicity, we also assume that $\Phi$ is a smooth function on $\S^2$.

First, we note a fact that for $u:\Omega\times\mathbb{R}^+\to \S^2$, the first equation of system \eqref{llgsp} is equivalent to
$$\p_t u=\frac{\al}{\al^2+1}\(\De u+|\n u|^2u-u\times (u\times(\tilde{h}+s))\)-\frac{1}{1+\al^2}u\times(h+s),$$
where
$$\tilde{h}=h_d(u) -\nabla_u\Phi.$$

Without loss of generality, we consider the following equivalent system
\begin{equation}\label{llgsp1}
\begin{cases}
\p_t u=\al\(\De u+|\n u|^2u-u\times (u\times(\tilde{h}+s))\)-u\times(h+s),\quad\quad&\text{(x,t)}\in\Om\times \Real^+,\\[1ex]
\p_t s=-\mbox{div}(A(u)\n s+u\otimes J_e)-D_0(x)\cdot s-D_0(x)\cdot s\times u, \quad\quad&\text{(x,t)}\in\Om\times \Real^+,
\end{cases}
\end{equation}
with the initial-boundary condition \eqref{in-bd} and $\al>0$.

\subsection{Related Work}
The LL equation is an important topic in both mathematics and physics, not only because it is closely related to the famous heat flow of harmonic maps and to the Schr\"{o}dinger flow on the sphere, but also because it has concrete physics background in the study of the magnetization in ferromagnets. In recent years, there has been trementdous interest in developing the well-posedness of LL equation and its related topics\cite{CJ,C-W,DWa,GW, JW1, JW2, M, PW, T, WC}. Here, we list only a few of results that are closely related to our work in the present paper.

First we recall some results on the weak solutions to the LL equation. The existence of weak solutions to LLG equation, also in the presence of magnetostrictive effects, was established by Visintin \cite{V} in 1985. P.L. Sulem, C. Sulem and C. Bardos in \cite{SSB} employed difference method to prove that the LL equation without dissipation term (that is Schr\"{o}dinger flow for maps into $\S^2$) defined on $\Real^n$ admits a global weak solution and a smooth local solution. For a bounded domain $\Om\subset\Real^3$, Alouges and Soyeur \cite{AS} showed the nonuniqueness of weak solutions to LLG equation.

Later, Y.D. Wang \cite{W} obtained the existence of weak solution to Schr\"odinger flow for maps from a closed Riemannian manifold into a two dimensional sphere by adopting a trick approximation equation, Galerkin method and then choosing suitable test functions to derive a priori estimates of $L^\infty$ on the approximate solutions. Tilioua \cite{T} (also see \cite{B}) also used the penalized method to show the global weak solution to the LLG equation with spin-polarized current. Recently, Z.L. Jia and Y.D. Wang \cite{JW1} (also see \cite{C-W}) employed a method originated from \cite{W}  to achieve the global weak solutions to a large class of LL flows in more general setting, where the base manifold is a bounded domain $\Om\subset\Real^n$($n\geq 3$) or Riemannian manifold $M^n$ and the target space is $\S^2$ or the unit sphere $\S^n_\g$ in a compact Lie algebra $\g$.

The first global well-posedness result for LL equation on $\mathbb{R}^n$ in critical spaces (precisely, global well-posedness for small data in the critical Besov spaces in dimensions $n\geqslant3$) was proved by Ionescu and Kenig in \cite{IK}, and independently by Bejanaru in \cite{B1}. This was later improved to global regularity for small data in the critical Sobolev spaces in dimensions $n\geqslant4$ in \cite{BIK}. Finally, in \cite{BIKT} the global well-posedness result for LL equation for small data in the critical Besov spaces in dimensions $n\geqslant2$ was addressed. In $\Real^n$ with $n\geq 3$, Melcher \cite{M} proved the existence, uniqueness and asymptotics of global smooth solutions for the LLG equation, valid under a smallness condition of initial gradients in the $L^n$ norm. His argument is based on deriving a covariant complex Ginzburg-Landau equation and using the Coulomb gauge.

Next, we retrospect some of the work related to local regular solutions of LL equation on bounded domains or compact manifolds. For the case LL equation without Gilbert damping term, one has shown the existence of local smooth solutions, we refer to \cite{DW1, DW*, SSB, PWW1, PWW2, ZGT}. For the case $\Om$ is a bounded domain in $\Real^3$, Carbou and Fabrie studied a nonlinear dissipative LL equation (i.e. $\al>0$) coupled with Maxwell equations in micromagnetism theory, and they proved the local existence and uniqueness of regular solutions for a so-called quasistatic model in \cite{CF}. Moreover, they showed global existence of regular solutions for small data in the 2D case for the LL equation (also see \cite{GH}). Recently, the local existence of very regular solution to LLG equation with electric current was addressed by applying the delicate Galerkin approximation method and adding compatibility initial-boundary condition in \cite{CJ}.

For the spin-magnetization system \eqref{llgsp} that takes into account the diffusion process of the accumulation, Garcia-Cervera and Wang \cite{GW} adopted Galerkin approximation and projection method to obtain the existence of weak solution. By using more refined harmonic analysis, X.K. Pu and W.D. Wang established the global existence and uniqueness of weak solutions to the simplified system \eqref{llgsp} from $\Real^2$ into $\S^2$ for large initial data in their paper \cite{PW}, where the partial regularity was shown. Recently, Z.L. Jia and Y.D. Wang \cite{JW2} employed a suitable auxiliary approximation equation and then take the Galerkin approximation method for the auxiliary equation as in \cite{W} to get the global weak solution of \eqref{llgsp}-\eqref{in-bd}. In particular, they also get the existence of weak solution to \eqref{llgsp}-\eqref{in-bd} without damping term (i.e., $\alpha=0$, the coupling system of Schr\"odinger flow and diffusion equation). It seems that there is few result on its regular solutions to this coupling system in the literature.

\subsection{Main results and Strategy} Inspired by the method used in \cite{CJ}, we show the locally very regular solution of LLGSP when the basis domain $\Om$ is a smooth bounded in $\Real^3$. Our main result can be stated as following theorems.

\begin{thm}\label{mth1}
Let $u_0\in H^{2}(\Om,\S^2)$ and $s_0\in H^{2}(\Om,\Real^3)$ satisfy the compatibility condition:
\begin{equation*}
\begin{cases}
\frac{\p u_0}{\p \nu}|_{\p \Om}=0,\\[1ex]
\frac{\p s_0}{\p \nu}|_{\p \Om}=0.
\end{cases}
\end{equation*}
 Suppose that $D_0\in C^2(\bar{\Om})$ and $D_0(x)\geq c_0>0$ for some constant $c_0$, $0<\theta<1$ and $J_e\in L^\infty(\Real^+,H^2(\Om))$. Then there exists $T^*>0$ depending only on the $H^2$-norm of $(u_0,s_0)$ such that \eqref{llgsp1}-\eqref{in-bd} admits a unique local solution $(u,s)$ for any $T<T^*$, which satisfies
\begin{enumerate}
	\item $|u|(x,t)=1$ in $[0,T]\times \Om$,
	\item  $(u,s)\in L^\infty([0,T], H^2(\Om,\Real^3))\cap L^2([0,T], H^3(\Om,\Real^3))$.
\end{enumerate}
\end{thm}
Without additional compatibility condition, we can get a more regular solution to \eqref{llgsp1} when $(u_0,s_0)\in H^{3}(\Om,\Real^3)$.
\begin{thm}\label{mth2}
Let $u_0\in H^{3}(\Om,\S^2)$ and $s_0\in H^{3}(\Om,\Real^3)$ satisfy the same compatibility condition as in Theorem \ref{mth1}. Suppose that $D_0\in C^3(\bar{\Om})$ and $D_0\geq c_0>0$, $0<\theta<1$, $J_e\in C^0(\Real^+,H^2(\Om))$ and $\p_t J_e\in L^2(\Real^+,H^1(\Om))$. If $(u,s)$ and  $T^*$ are respectively the solution and the existence time given in Theorem \ref{mth1}, then, for any $T<T^*$, there holds
$$(u,s)\in L^\infty([0,T], H^3(\Om,\Real^3))\cap L^2([0,T], H^4(\Om,\Real^3)).$$
\end{thm}

In general, we can show the very regular solution to \eqref{llgsp1} under adding high order compatibility conditions.
\begin{thm}\label{mth3}
Let $k\geq 4$, $u_0\in H^{k}(\Om,\S^2)$ and $s_0\in H^{k}(\Om,\Real^3)$ satisfy the compatibility condition at $[\frac{k}{2}]-1$ order, which is given in the definition \ref{def-comp}. Let $(u,s)$ and $T^*>0$ be the same as in Theorem \ref{mth1}. In addition, we assume that $D_0\in C^{k}(\bar{\Om})$, $D_0\geq c_0>0$, and for any $i\leq k-[\frac{k}{2}]-1$ there holds
$$\p^i_tJ_e\in C^0(\Real^+,H^{2[\frac{k}{2}]-2i}(\Om,\Real^3))\cap L^2(\Real^+,H^{2[\frac{k}{2}]-2i+1}(\Om,\Real^3)).$$
Then, for any $T<T^*$, we have
$$(u,s)\in L^\infty([0,T], H^k(\Om,\Real^3))\cap L^2([0,T], H^{k+1}(\Om,\Real^3)).$$
\end{thm}

Theorem \ref{mth1} is achieved by choosing a suitable auxiliary equation of spin equation with respect to $s$ (i.e. the second equation of \eqref{llgsp1}) and then applying Galerkin approximation to the modified equation of \eqref{llgsp1} and estimating some suitable energies directly. However, we cannot improve the regularity of strong solution $$(u,s)\in L^\infty([0,T], H^2(\Om,\Real^3))\cap L^2([0,T], H^3(\Om,\Real^3))$$ obtained in Theorem \ref{mth1} by getting higher order energy estimates, since the right hand sides of system \eqref{llgsp1} do not satisfy the homogeneous Neumann boundary condition. To proceed, following Carbou's idea in \cite{CJ}, we consider the differential of Galerkin approximation to system \eqref{llgsp1} with respect to time and prove by the same way that $(\p_tu,\p_ts)\in L^\infty([0,T], H^1(\Om,\Real^3))\cap L^2([0,T], H^2(\Om,\Real^3))$, where $0<T<T^*$. Thus, we can show Theorem \ref{mth2} by a bootstrap argument using equation \eqref{llgsp}.

However, we cannot enhance the regularity of solution $(u,s)$ in Theorem \ref{mth2} by applying directly the second order differential of Galerkin approximation to system \eqref{llgsp1}, because of the lower regularity of Galerkin projection, although the Galerkin projection is the crucial analytic tool in our argument (see Lemma \ref{es-P_n}).  Generally, to improve the regularity of solution $(u,s)$, we need to impose so-called compatibility conditions of initial data. Let $k\geq 1$, Considering the equation of $(\p^k_t u, \p^k_t s)$ with compatibility condition \eqref{com-cond}, the very regular solution to system \eqref{llgsp1} can be shown. More precisely, we use the simplest case to explain the strategy of enhancing regularity $\P$:
\begin{itemize}
	\item[$(1)$] Assume $(u_0,s_0)\in H^3(\Om,\Real^3)$. By considering the differential of the Galerkin approximation of \eqref{llgsp1} with respect to time $t$, we can show
	\[(u,s)\in  L^\infty([0,T],H^3(\Om,\Real^3))\cap L^2([0,T],H^4(\Om,\Real^3)).\]
	\item[$(2)$] Assuming $(u_0,s_0)\in H^4(\Om,\Real^3)$, we can get a regular solution \[(u_1,s_1)\in L^\infty([0,T],H^2(\Om,\Real^3))\cap L^2([0,T],H^3(\Om,\Real^3))\]
	 to the equation \eqref{comp-llgsp} of $(\p_t u,\p_t s)$. For this step, we need the compatibility condition at the boundary for $(\p_t u,\p_t s)$ when $t=0$.
	\item[$(3)$] The uniqueness guarantees $(\p_t u,\p_t s)=(u_1,s_1)$. It implies
	\[(u,s)\in L^\infty([0,T],H^4(\Om,\Real^3))\cap L^2([0,T],H^5(\Om,\Real^3)).\]
	 by a bootstrap argument using the equation \eqref{llgsp} again.
	\item[$(4)$] Assume $(u_0,s_0)\in H^5(\Om,\Real^3)$, we obtain
	\[(u,s)\in L^\infty([0,T],H^5(\Om,\Real^3))\cap L^2([0,T],H^6(\Om,\Real^3))\]
	by repeating the argument in $(1)$ to $(\p_t u,\p_ts)$.
\end{itemize}
Here $0<T<T^*$. For higher order regularity, to add higher order compatibility conditions we can prove Theorem \ref{mth3} by repeating the above process $\P$.

The rest of our paper is organized as follows. In section \ref{s: pre}, we introduce the basic notations in Sobolev space and some critical preliminary lemmas. Meanwhile the the compatibility condition of initial data to system \eqref{llgsp1} will also be given. In section \ref{s: reg-sol}, we prove Theorem \ref{mth1} by employing Galerkin approximation method. Theorem \ref{mth2} and Theorem \ref{mth3} will be built up in subsection \ref{ss: 1-reg} and \ref{ss: main-thm} respectively.

\section{Preliminary}\label{s: pre}
\subsection{Notations}\label{ss: pre-lem}\

In this subsection, we first recall some notations on Sobolev spaces which will be used in whole context. Let $u=(u_1,u_2,u_3):\Om\to\S^{2}\hookrightarrow\Real^3$ be a map. We set
$$H^{k}(\Om,\S^{2})=\{u\in W^{k,2}(\Om,\Real^{3}):|u|=1\,\,\text{for a.e. x}\in \Om\}$$
and
\[W^{k,l}_{2}(\Om\times[0,T],\S^{2})=\{u\in W^{k,l}_{2}(\Om\times[0,T],\Real^{3}):|u|=1\,\,\text{for a.e. (x,t)}\in \Om\times[0,T]\},\]
for $k,\,l\in \mathbb{N}$ and denote $H^0(\Om,\Real^3)=L^2(\Om,\Real^2)$.

Moreover, let $(B,\norm{\cdot}_B)$ be a Banach space and $f:[0,T]\to B$ be a map. For any $p>0$ and $T>0$, we define
\[\norm{f}_{L^p([0,T], B)}:=\(\int_{0}^{T}\norm{f}^p_{B}dt\)^{\frac{1}{p}},\]
and set
\[L^p([0,T],B):=\{f:[0,T]\to B:\norm{f}_{L^p([0,T],B)}<\infty\}.\]
In particular, we denote
\[L^{p}([0,T],H^{k}(\Om,\S^{2}))=\{u\in L^{p}([0,T],H^{k}(\Om,\Real^{3})):|u|=1\,\,\text{for a.e. (x,t)}\in \Om\times[0,T]\},\]
where $k\in \mathbb{N}$ and $p\geq 1$.

\subsection{Estimates of $h_d$ and some lemmas}\

For later application, we recall some regular results. Let $u:\Om \to \Real^3$ be a map. Then, in the sense of distributions, the induced vector field is defined by
\[h_d(u):=-\n\int_\Om\n N(x-y)u(u)\,dx,\]
where $N(x)=-\frac{1}{4\pi|x|}$ is the Newton potential on $\Real^3$. Hence, the following estimates of $h_d$ is a fundamental result in theory of singular integral operators, its proof can be found in \cite{CF,CJ, C-W}.

\begin{lem}\label{es-h_d}
	Let $p\in(1,\infty)$ and $\Om$ be a bounded smooth domain in $\Real^{3}$. Assume that $u \in W^{k,p}(\Om,\Real^{3})$ for $k\in \mathbb{N}$. Then, the restriction of $h_d(u)$ to $\Om$ belongs to $W^{k,p}(\Om, \Real^{3})$. Moreover, there exists constants $C_{k,p}$, which is independent of $u$, such that
	$$\lVert h_d(u)\rVert_{W^{k,p}(\Om)}\leq C_{k,p}\lVert u\rVert_{W^{k,p}(\Om)}.$$
	In fact, $h_{d}:W^{k,p}(\Om, \Real^{3})\to W^{k,p}(\Om, \Real^{3})$ is a linear bounded operator.
\end{lem}

The $L^2$ theory of Laplace operator with Neumann boundary condition implies the following Lemma of  equivalent norm, to see \cite{C-W, Weh}.
\begin{lem}\label{eq-norm}
	Let $\Om$ be a bounded smooth domain in $\Real^{m}$ and $k\in \mathbb{N}$. There exists a constant $C_{k,m}$ such that, for all $u\in H^{k+2}(\Om)$ with $\frac{\p u}{\p \nu}|_{\p\Om}=0$,
	\begin{equation}\label{eq-n}
	\norm{u}_{H^{2+k}(\Om)}\leq C_{k,m}(\norm{u}_{L^{2}(\Om)}+\norm{\De u}_{H^{k}(\Om)}).
	\end{equation}	
\end{lem}
In particular, we define the $H^{k+2}$-norm of $u$ as follows
$$\norm{u}_{H^{k+2}(\Om)}:=\norm{u}_{L^2(\Om)}+\norm{\De u}_{H^k(\Om)}.$$
We also need some basic lemmas on the space $H^m(\Om)$ with $m\geq 1$ below ( cf. \cite{CJ}).
\begin{lem}\label{algebra1}
Assume $\Om$ be a bounded smooth domain in $\Real^3$. Let $f\in H^1(\Om)$ and $g\in H^m(\Om)$ with $m\geq2$. Then $f\cdot g\in H^1(\Om)$.
\end{lem}

Applying the fact $H^{2}(\Om)\subset W^{1,6}(\Om)\subset L^\infty(\Om)$, the following result can be obtained from the above lemma directly.
\begin{lem}\label{algebra2}
Assume $\Om$ be a bounded smooth domain in $\Real^3$. Let $f$ and $g$ in $H^m(\Om)$ with $m\geq 2$, there holds
$$f\cdot g\in H^m(\Om).$$
In fact, $(H^m(\Om),\cdot)$ is an algebra.
\end{lem}
\subsection{Comparison theorem for ODE and Aubin-Simon's compactness}\

In order to show the uniform estimates and the convergence of solutions to the approximated equation constructed in coming sections, we need to use the comparison theorem for ordinary differential equation(ODE) and the classical compactness result in \cite{Sim}.

\begin{lem}\label{c-thm}
Let $f: \Real^+\times \Real\to \Real$ be a continuous function, which is locally Lipschitz on the second variable. Let $z:[0,T^*)\to \Real$ be the maximal solution of the Cauchy problem:
\begin{equation*}
\begin{cases}
z^\prime=f(t,z),\\[1ex]
z(0)=z_0.	
\end{cases}
\end{equation*}
Let $y:\Real^+\to \Real$ be a $C^1$ function such that
\begin{equation*}
\begin{cases}
y^\prime\leq f(t,y),\\[1ex]
y(0)\leq z_0.	
\end{cases}
\end{equation*}
Then,
$$y(t)\leq z(t),\quad\text{t}\in [0,T^*).$$
\end{lem}

\begin{lem}[Aubin-Simon compactness Lemma]\label{A-S}
	Let $X\subset B\subset Y$ be Banach spaces with compact embedding $X\hookrightarrow B$, $F$ is a bounded set in $L^{q}([0,T],B)$ for $q>1$.
	If $F$ is bounded in $L^{1}([0,T],X)$ and $\frac{\p F}{\p t}$ is bounded in $L^{1}([0,T],Y)$, where $\frac{\p F}{\p t}=\{\frac{\p f}{\p t}:f\in F\}$, then $F$ is relatively compact in $L^{p}([0,T],B)$, for any $1\leq p<q$.
\end{lem}
\subsection{Galerkin basis and Galerkin projection}\label{Ga-basis}\

Let $\Om$ be a bounded smooth domain in $\Real^m$, $\ld_{i}$ be the $i^{th}$ eigenvalue of the operator $\De-I$  with Neumann boundary condition, whose corresponding eigenfunction is $f_{i}$. That is,
$$(\De-I)f_{i}=-\ld_{i}f_{i}\,\,\,\,\quad\text{with}\quad\,\,\,\,\frac{\p f_{i}}{\p \nu}|_{\p\Om}=0.$$

Without loss of generality, we assume $\{f_{i}\}_{i=1}^{\infty}$ are completely standard orthonormal basis of $L^{2}(\Om,\Real^{1})$. Let $H_{n}=span\{f_{1},\dots f_{n}\}$ be a finite subspace of $L^{2}$, $P_{n}:L^{2}(\Om, \Real^1)\to H_{n}$ be the canonical projection. In fact, for any $f\in L^{2}$,
\[f^{n}=P_{n}f=\sum_{i=1}^{n}\<f,f_{i}\>_{L^{2}}f_{i}\quad\text{and}\quad \lim_{n\to \infty}\norm{f-f^{n}}_{L^{2}}=0.\]

For this canonical projection, we have the following uniform estimates, which is essential for our method to get very regularity of solution in section \ref{s: very-reg-sol}. Its proof can be found in \cite{CJ}.
\begin{lem}\label{es-P_n}
	There exists a constant $C$ such that for all $n$, the projection $P_n$ satisfies the following properties:
	\begin{enumerate}
		\item For all $f\in H^1(\Om,\Real^1)$, $\norm{P_n(f)}_{H^1(\Om)}\leq \norm{f}_{H^1(\Om)}$,
		\item For all $f\in H^2(\Om,\Real^1)$ such that $\frac{\p f}{\p \nu}|_{\p \Om}=0$, $\norm{P_n(f)}_{H^2(\Om)}\leq C\norm{f}_{H^2(\Om)}$,
		\item For all $f\in H^3(\Om,\Real^1)$ such that $\frac{\p f}{\p \nu}|_{\p \Om}=0$, $\norm{P_n(f)}_{H^3(\Om)}\leq C\norm{f}_{H^3(\Om)}$.
	\end{enumerate}
\end{lem}
\begin{rem}
	Unfortunately, we can't get such estimates for $f\in H^m(\Om,\Real^1)$ such that $\frac{\p f}{\p \nu}|_{\p \Om}=0$, when $m\geq 4$.
\end{rem}

\subsection{Compatibility conditions of the initial data}\label{ss: comp-cond}\

In order to get higher regularity of the solution to system \eqref{llgsp1} in Theorem \ref{mth3}, we now turn  to define the compatibility conditions on the initial data. To clearly clarify the ideas of deriving compatibility conditions from equation \eqref{llgsp1}, first of all we assume that $(u,s)$ is a smooth solution.

Let $(u_k,s_k)=(\p^k_t u,\p^k_t s)$ with $k\geq 1$. Then a tedious but direct calculation shows $(u_k,s_k)$ satisfies the following equation
\begin{equation}\label{comp-llgsp3}
\begin{cases}
\p_t u_k=\al \De u_k-u\times \De u_k+K_k(\n u_k,\n s_k)+L_k(u_k, s_k)+F_k(u,s)&\text{(x,t)}\in\Om\times \Real^+,\\[1ex]
\p_t s_k=-\mbox{div}\(A(u)\n s_k\)+Q_k(\n u_k,\n s_k)+T_k(u_k,s_k)+Z_k(u,s)&\text{(x,t)}\in\Om\times \Real^+,
%\frac{\p u_k}{\p \nu}=0\quad\quad&\text{(x,t)}\in\p \Om\times \Real^+,\\[1ex]
%\frac{\p s_k}{\p \nu}=0\quad\quad&\text{(x,t)}\in\p \Om\times \Real^+.
\end{cases}
\end{equation}
where
\begin{equation*}
\begin{aligned}
K_k(\n u_k,\n s_k)=&2\al (\n u_k\cdot \n u) u,\\
Q_k(\n u_k,\n s_k)=&-\mbox{div}(u_k\otimes J_e),
\end{aligned}
\end{equation*}
\begin{equation*}
\begin{aligned}
L_k(u_k,s_k)=&\al |\n u|^2u_k-\al u_k\times \(u\times(\tilde{h}(u)+s)\)-\al u\times \(u_k\times(\tilde{h}(u)+s)\)\\
&-\al u\times \(u\times(h_d(u_k)-\n^2\Phi(u)\cdot u_k+s_k)\)\\
&-u\times\(h_d(u_k)-\n^2\Phi(u)\cdot u_k+s_k\)-u_k\times\(h(u)+s\),\\
F_k(u,s)=&\sum_{i+j+l=k,\,0\leq i,j,l<k}\n u_i\#\n u_j\#u_l+\sum_{i+j+l=k,0\leq i,j,l<k} u_i\#u_j\#(\bar{h}(u_l)+s_l)\\
&+\sum_{i+j=k,\,0\leq i,j<k}u_i\#(\bar{h}(u_j)+s_j)+\sum_{i+j=k,\,0\leq i,j<k}u_i\#\De u_j+R_k,\\
T_k(u_k,s_k)=&\th \mbox{div}(D_0u_k\otimes(\n s\cdot u+D_0u\otimes(\n s\cdot u_k)))-D_0s_k-D_0s_k\times u-D_0 s\times u_k,\\
Z_k(u,s)=&\mbox{div}\(D_0\th \sum_{i+j+l=k,\,i,j,l<k} u_i\#\n s_j\#u_l\)\\
&+\sum_{i+j=k,\,i,j<k}D_0s_i\# u_j-\sum_{i+j=k,\,i<k}\mbox{div}(u_{i}\#\p^j_tJ_e).
\end{aligned}
\end{equation*}
Here
\begin{align*}
\bar{h}(u_l)=&h_d(u_l)-\n^{2}\Phi(u)\cdot u_l+\sum_{j_1+\dots+j_i=l,\,i>1}\n^{i+1}_u\Phi(u)\#u_{j_1}\dots\#u_{j_i},\\
R_k=&u\times (u\times \sum_{j_1+\dots+j_i=k,\,i>1}\n^{i+1}_u\Phi(u)\#u_{j_1}\dots\#u_{j_i})\\
&+u\times \sum_{j_1+\dots+j_i=k,\,i>1}\n^{i+1}_u\Phi(u)\#u_{j_1}\dots \#u_{j_i},
\end{align*}
and $\#$ denotes the linear contraction.

Its initial data is:
\begin{equation}\left\{
\begin{aligned}
V_k(u_0,s_0)=u_k(x,0)=&\al \De V_{k-1}-\al V_0\times \De V_{k-1}+K_{k-1}(\n V_{k-1},\n W_{k-1})\\
&+L_{k-1}(V_{k-1},W_{k-1})+F_{k-1},\\
W_k(u_0,s_0)=s_k(x,0)=&-\mbox{div}\(A(V_0)\n W_{k-1}\)+Q_{k-1}(\n V_{k-1},\n W_{k-1})\\
&+T_{k-1}(V_{k-1},W_{k-1})+Z_{k-1}.
\end{aligned}\right.
\end{equation}
Here $(u_l,s_l)$ has been replaced by $(V_l,W_l)$ in the terms $K_{k-1}$, $L_{k-1}$, $F_{k-1}$, $Q_{k-1}$, $T_{k-1}$ and $Z_{k-1}$, and $(V_0,W_0)=(u_0,s_0)$.

In particular, we have
\begin{equation*}
\begin{cases}
V_1=\al\(\De u_0+|\n u_0|^2u_0-u_0\times (u_0\times(\tilde{h}(u_0)+s_0))\)-u_0\times(\De u_0 +\tilde{h}(u_0)+s_0),\\
W_1=-\mbox{div} J_s(u_0,s_0)-D_0(x)\cdot s_0-D_0(x)\cdot s_0\times u_0.
\end{cases}
\end{equation*}
Now, we are in the position to state the compatibility conditions on initial data $(u_0,s_0)$, associated to equation \eqref{llgsp1}, as follows.
\begin{defn}\label{def-comp}
Let $k\in \mathbb{N}$, $(u_0,s_0)\in H^{2k+1}(\Om,\Real^3)$ and $\p^i_tJ_e(x,0)\in H^1(\Om,\Real^3)$ for $0\leq i\leq k$. We say $(u_0,s_0)$ satisfies the compatibility condition at order $k$, if for any $j\in\{0,1,\dots,k\}$, there holds
\begin{equation}\label{com-cond}
\begin{cases}
\frac{\p V_j}{\p \nu}|_{\p \Om}=0,\\[1ex]
\frac{\p W_j}{\p \nu}|_{\p \Om}=0.
\end{cases}
\end{equation}
\end{defn}
\medskip
Intrinsically, we denote
\[\tau_{\Phi, s}(u)=\tau(u)-u\times(u\times (\tilde{h}+s)),\]
where $\tau(u)=\De u+|\n u|^2 u$ be the tension field. Then the first equation in \eqref{llgsp1} becomes
\[\p_t u=\al \tau_{\Phi,s}(u)-u\times \tau_{\Phi,s}(u).\]
And hence, after taking k times derivatives at direction $t$ for the above equation, $u_k$ satisfies the following equation
\[\p_t u_k=\p^k_t\p_t u=\al\p^k_t\tau_{\Phi,s}(u)-\sum_{i+j=k}C^{i}_{k}\p^i_tu\times \p^j_t\tau_{\Phi,s}(u),\]
where $C^i_{k}=\frac{i!}{k!(k-i)!}$. When $k\geq 1$, let $\tilde{V}_k=\p^{k-1}_t\tau_{\Phi,s}(u)(x,0)$, there holds
\[V_k=\al\tilde{V}_k-u_0\times \tilde{V}_k+\sum_{i+j=k, i\geq1}C^{i}_{k}V_i\times \tilde{V}_j,\]
for the sake of convenience, where we denote $\tilde{V}_0=V_0=u_0$. In particular, there holds
\[\tilde{V}_1=\tau(u_0)-u_0\times (u_0\times(\tilde{h}(u_0)+s_0)).\]
Therefore, it is not difficult to show that the $k$ order compatibility condition defined by \ref{def-comp} has the below equivalent characterization.
\begin{prop}
Let $k\in \mathbb{N}$, $(u_0,s_0)\in H^{2k+1}(\Om,\Real^3)$ and $\p^i_tJ_e(x,0)\in H^1(\Om,\Real^3)$ for $0\leq i\leq k$. $(u_0,s_0)$ satisfies the compatibility condition at k order if and only if for any $j\in\{0,1,\dots,k\}$, there holds
\begin{equation}\label{com-cond1}
\begin{cases}
\frac{\p \tilde{V}_j}{\p \nu}|_{\p \Om}=0,\\[1ex]
\frac{\p W_j}{\p \nu}|_{\p \Om}=0.
\end{cases}
\end{equation}
\end{prop}
\medskip
\section{Regular solution}\label{s: reg-sol}
In this section, we consider the existence of short-time regular solution to \eqref{llgsp1}. To this end, we adopt the following equivalent equation
\begin{equation}\label{llgsp2}
\begin{cases}
\p_t u=\al\(\De u+|\n u|^2u-u\times (u\times(\tilde{h}+s))\)-u\times(h+s),\quad&\text{(x,t)}\in\Om\times \Real^+,\\[1ex]
\p_t s=-\mbox{div}\(A(\J(u))\n s+u\otimes J_e\)-D_0(x)\cdot s-D_0(x)\cdot s\times u, \quad&\text{(x,t)}\in\Om\times \Real^+,
\end{cases}
\end{equation}
with the initial-boundary condition \eqref{in-bd} and $\al>0$. Here,
\[-A(u)=D_0
\begin{pmatrix}
	1-\th u^2_1&-\th u_1u_2&-\th u_1u_3\\
	-\th u_2u_1&1-\th u^2_2&-\th u_2u_3\\
	-\th u_3u_1&-\th u_3u_2&1-\th u^2_3\\
\end{pmatrix}
\]
is a positively definite matrix of functions with
\[0<(1-\th|u|^2)D_0|\xi|^2\leq -\xi^T A(u)\xi\leq D_0|\xi|^2\]
for any vector $\xi$ in $\Real^3$ if $\th|u|^2< 1$,
$$\J(u)=\frac{\sqrt{1+\de}}{\sqrt{\de+|u|^2}}u$$
with $\de>0$ to be determined later, and
$$\tilde{h}(u)=h_d(u)-\n_u\Phi(\J(u)).$$

It should also be pointed out that the above $\Phi(u)$ has been extended to $\overline{B}_{\sqrt{1+\de}}(0)\subset \Real^3$. In fact, we extend $\Phi(z)$ by
 \begin{equation*}
 \tilde{\Phi}(z)=
 \left\{
 \begin{aligned}
 &\zeta(|z|^2)\Phi\left(\frac{z}{|z|}\right),\,\, &|z|^2>\delta_0,\\
 &0, \,\, &|z|^2\leq\delta_0,
 \end{aligned}\right.
 \end{equation*}
 where $\zeta(t):[0, 2]\to [0, 1]$ is a $C^\infty$-smooth function with $\zeta(t)\equiv 0$ on $[0, \, 2\delta_0]$ ($2\de_0<1$) and $\zeta(t)=1$ on $[1,2]$. It is easy to see that $\tilde{\Phi}$ is $C^\infty$-smooth on $\overline{B}_{\sqrt{1+\de}}(0)$. For the sake of simplicity, we still denote $\tilde{\Phi}$ by $\Phi$.

\subsection{Galerkin approximation and a priori estimates}\label{ss: app-estm}\

Let $H_n$ be the $n$-dimensional subspace of $L^2(\Om)$ defined in Section \ref{Ga-basis}, $P_n$ be the Galerkin projection. Next, we seek a solution $(u^{n}, s^n)$ in $H_{n}$ to the following Galerkin approximation equation associated to \eqref{llgsp2}, i.e.
\begin{equation}\label{app-llgsp}
\begin{cases}
\p_t u^n= P_n\(\al(\De u^n+|\n u^n|^2u^n-u^n\times (u^n\times(\tilde{h}+s^n)))-u^n\times(h+s^n)\),\\[1ex]%&\text{(x,t)}\in\Om\times \Real^+,\\[1ex]
\p_t s^n=P_n\(-\mbox{div}\(A(\J(u^n))\n s^n+u^n\times J_e\)-D_0(x)\cdot s^n-D_0(x)\cdot s^n\times u^n\), %&\text{(x,t)}\in\Om\times \Real^+,
\end{cases}
\end{equation}
with initial date $(u^n(\cdot,0),s^n(\cdot,0))=(u^n_0,s^n_0)$.

Let $u^{n}=\sum_{1}^{n}g^{n}_{i}(t)f_{i}(x)$ and $s^{n}=\sum_{1}^{n}\ga^{n}_{i}(t)f_{i}(x)$, $G^{n}(t)=\{g^{n}_{1}(t)\dots g^{n}_{n}(t),\ga^n_1(t)\dots \ga^n_n(t)\}$ be a vector-valued function. Then, a direct calculation shows that $G^{n}(t)$ satisfies the following ordinary differential equation
\begin{equation}\left\{
\begin{aligned}
\frac{\p G^{n}}{\p t}&=F(t,G^n),\\
G^{n}(0)&=(\<u_{0},f_{1}\>,\dots,\<u_{0},f_{n}\>,\<s_{0},f_{1}\>,\dots,\<s_{0},f_{n}\>),
\end{aligned}\right.
\end{equation}
where $F(G^n)$ is locally Lipschitz continuous with respect to  $G^n$, since $\J(f)$ is locally Lipschitz on $f$. Hence, there exist a solution $(u^{n},s^n)$ to \eqref{app-llgsp} on $\Om\times[0,T^n_{0})$ for some $T^n_{0}>0$.

If we choose $(u^n, s^n)$ as a test function to multiply the two sides of equation \eqref{app-llgsp}, then it is easy to see that there hold
\begin{equation}\label{L^2-u}
\frac{1}{2}\frac{\p}{\p t}\int_{\Om}|u^n|^2dx+\al\int_{\Om}|\n u^n|^2dx=\al\int_{\Om}|\n u^n|^2|u^n|^2dx
\end{equation}
and
\begin{equation}\label{L^2-s}
\begin{aligned}
&\frac{1}{2}\frac{\p}{\p t}\int_{\Om}|s^n|^2dx+(1-\th(1+\de))\int_{\Om}D_0|\n s^n|^2dx\\
\leq& \int_{\Om}\<u^n\otimes J_e,\n s^n\>dx-\int_{\p \Om}u^n\cdot s^n\<J_e,\nu\>d\mu_{\p \Om}\\
\leq &C(\ep,c_0)|u^n|^2_{L^{\infty}}(\int_{\Om}|J_e|^2dx+\int_{\p\Om}|J_e|^2d\mu_{\p \Om})\\
&+\ep(1+\frac{\tilde{C}}{c_0})\int_{\Om}D_0|\n s^n|^2dx+\tilde{C}\ep \int_{\Om}|s^n|^2dx.
\end{aligned}
\end{equation}
Here, we have used the following fact
\begin{align*}
	\int_{\Om}\<A(\J(u^n))\n s^n,\n s^n\>dx&=\int_{\Om}D_0|\n s^n|^2dx-(1+\de)\th\int_{\Om}D_0\frac{\<\n s^n,u^n\>^2}{\de+|u^n|^2}dx\\
	                                                                 &\geq (1-(1+\de)\th)\int_{\Om} D_0|\n s^n|^2dx,
\end{align*}
and the Trace theorem to derive
\begin{align*}
\int_{\p\Om}|s^n|^2dx\leq \tilde{C}(\int_{\Om}|\n s^n|^2dx+\int_{\Om}|s^n|^2dx),
\end{align*}
$$\int_{\p\Om}|J_e|^2dx\leq \tilde{C}(\int_{\Om}|\n J_e|^2dx+\int_{\Om}|J_e|^2dx).$$
Therefore, by choosing $\de<1-\frac{1}{\th}$ and then suitable $\ep>0$, we have
\begin{equation}\label{L^2-us}
\begin{aligned}
&\frac{1}{2}\frac{\p}{\p t}\int_{\Om}|u^n|^2+|s^n|^2dx+\al\int_{\Om}|\n u^n|^2dx+\frac{1}{2}(1-\th(1+\de))\int_{\Om}D_0|\n s^n|^2dx\\
\leq &\al|u^n|^2_{L^\infty}\int_{\Om}|\n u^n|^2dx+C(\de,\th, c_0)\norm{J_e}^2_{H^1}|u^n|^2_{L^{\infty}}+C(\de,\th,c_0) \int_{\Om}|s^n|^2dx\\
\leq &C(\al, \de, \th, c_0)(1+\norm{J_e}^2_{H^1}) (U^2+U+S).
\end{aligned}
\end{equation}
Here we denote $U=\norm{u^n}^2_{H^2}$ and $S=\norm{s^n}^2_{H^2}$.
\medskip

In order to get the $H^{3}$-energy estimates, we choose $(v,w)=(\De^2 u^{n},\De^2 s^n)$ as the test function. By using integration by parts, we take a simple computation to derive
\begin{equation}\label{H^3-u}
\begin{aligned}
&\frac{1}{2}\frac{\p}{\p t}\int_{\Om}|\De u^n|^2dx+\al\int_{\Om}|\n \De u^n|^2dx\\
=&-\al\int_{\Om}\<\n(|\n u^n|^2u^n),\n \De u^n\>dx+\al\int_{\Om}\<\n(u^n\times(u^n\times(\tilde{h}+s^n))),\n \De u^n\>dx\\
&+\int_{\Om}\<\n(u^n\times(h+s^n)),\n \De u^n\>dx\\
= &I+II+III
\end{aligned}
\end{equation}
and
\begin{equation}\label{H^3-s}
\begin{aligned}
\frac{1}{2}\frac{\p}{\p t}\int_{\Om}|\De s^n|^2dx
=&-\int_{\Om}\<\mbox{div}\(A(\J(u^n))\n s^n+u^n\otimes J_e\),\De^2 s^n\>dx\\
&+\int_{\Om}\<\n(D_0\cdot s^n+D_0\cdot s^n\times u^n),\n\De s^n\>dx\\
=&IV+V.
\end{aligned}
\end{equation}

By direct calculations, we have
\begin{equation}
\begin{aligned}\label{I}
	|I|=&\al|\int_{\Om}\<\n(|\n u^n|^2u^n),\n \De u^n\>dx|\\
\leq& 2\al\int_{\Om}|\n^2 u^n||\n u^n||u^n||\n \De u^n|dx + 2\al\int_{\Om}|\n u^n|^3|\n \De u^n|dx\\
	\leq& 2\al\norm{u^n}_{L^\infty}\norm{\n u^n}_{L^6}\norm{\n^2u^n}_{L^3}\norm{\n \De u^n}_{L^2}+2\al\norm{\n u^n}^3_{L^{6}}\norm{\n \De u^n}_{L^2}\\
	=&2\al (I_1+I_2).
\end{aligned}
\end{equation}
For the above two terms $I_1$ and $I_2$, there hold
\begin{align*}
I_1=&\norm{u^n}_{L^\infty}\norm{\n u^n}_{L^6}\norm{\n^2u^n}_{L^3}\norm{\n \De u^n}_{L^2}\\
\leq &C\norm{u^n}^2_{H^2}\norm{\n^2 u}^\frac{1}{2}_{L^2}\norm{\n^2 u}^\frac{1}{2}_{L^6}\norm{\n \De u^n}_{L^2}\\
\leq& C\norm{u^n}^{2+1/2}_{H^2}\norm{u^n}^\frac{1}{2}_{H^3}\norm{\n \De u^n}_{L^2}\\
\leq& C\norm{u^n}^{3}_{H^2}\norm{\n \De u^n}_{L^2}+C\norm{u^n}^{2+1/2}_{H^2}\norm{\n \De u^n}^{3/2}_{L^2}\\
\leq& \ep \norm{\n \De u^n}^{2}_{L^2} +C(\ep)(\norm{u^n}^{6}_{H^2}+\norm{u^n}^{10}_{H^2})\\
\leq&\ep \norm{\n \De u^n}^{2}_{L^2} +C(\ep)(U^3+U^5)
\end{align*}
and
\begin{align*}
I_2=&\norm{\n u^n}^3_{L^{6}}\norm{\n \De u^n}_{L^2}\\
\leq &C(\ep)\norm{u^n}^6_{H^{2}}+\ep \norm{\n \De u^n}^2_{L^2}\\
\leq &C(\ep)U^3+\ep \norm{\n \De u^n}^2_{L^2}.
\end{align*}
Here we have used the following facts
$$H^{2}(\Om)\hookrightarrow W^{1,6}(\Om)\hookrightarrow L^{\infty}(\Om),$$
and
$$\norm{f}_{L^{3}}(\Om)\leq \norm{f}^{1/2}_{L^2(\Om)}\norm{f}^{1/2}_{L^6(\Om)}.$$

For the term $II$, there holds
\begin{equation}\label{II}
\begin{aligned}
|II|\leq& \int_{\Om}|\n u^n||u^n|(|h_d(u^n)|+|\n_u\Phi(\J(u^n))|+|s^n|)|\n \De u^n|dx\\
&+\int_{\Om}|u^n|^2(|\n h_d(u^n)| + |\n^2_u\Phi(\J(u^n))||\n \J(u^n)|+|\n s^n|)|\n \De u^n|dx\\
=&II_1+II_2.
\end{aligned}
\end{equation}

By direct calculations, we have
\begin{align*}
II_1\leq& \norm{u^n}_{L^\infty}\norm{\n u^n}_{L^6}(\norm{h_d(u^n)}_{L^{3}}+\norm{\n_u \Phi(\J(u^n))}_{L^3}+\norm{s^n}_{L^3})\norm{\n\De u^n}_{L^2}\\
\leq& C(\ep, \Phi)\norm{u^n}^4_{H^2}(\norm{u^n}^2_{H^2}+\norm{s^n}^2_{H^1}+1)+\ep\norm{\n \De u^n}^2_{L^2}\\
\leq&C(\ep, \Phi)U^2(U+S+1)+\ep\norm{\n \De u^n}^2_{L^2},
\end{align*}
and
\begin{align*}
II_2\leq & \norm{u^n}^2_{L^\infty}(\norm{\n h_d(u^n)}_{L^2}+|\n^2_{u}\Phi(\J(u^n))|_{L^\infty}\sqrt{1+1/\de}\norm{\n u^n}_{L^2}\\
&+\norm{\n s^n}_{L^2})\norm{\n \De u^n}_{L^2}\\
\leq &C(\ep, \de, \Phi)\norm{u^n}^4_{H^2}(\norm{u^n}^2_{H^1}+\norm{s^n}^2_{H^1})+\ep\norm{\n \De u^n}^2_{L^2}\\
\leq &C(\ep, \de, \Phi)U^2(U+S)+\ep\norm{\n \De u^n}^2_{L^2}.
\end{align*}

For the term $III$, we can show
\begin{equation}\label{III}
\begin{aligned}
|III|\leq& \int_{\Om}|\n u^n|(|\De u^n|+|h_d(u^n)|+|\n_u\Phi(\J(u^n))|+|s^n|)|\n \De u^n|dx\\
&+\int_{\Om}|u^n|(|\n h_d(u^n)|+|\n^2\Phi(\J(u^n))||\n u^n|+|\n s^n|)|\n \De u^n|dx\\
\leq &C(\ep, \de, \Phi)(\norm{u^n}^6_{H^2}+\norm{u^n}^4_{H^2}+\norm{u^n}^2_{H^2}\norm{s^n}^2_{H^2}+1)+2\ep \norm{\n \De u^n}^2_{L^2}\\
\leq& C(\ep, \de, \Phi)(U^3+U^2+US+1)+2\ep \norm{\n \De u^n}^2_{L^2}
\end{aligned}
\end{equation}
Here, we have canceled the term $\<u\times\n\De u, \n\De u\>=0$.

By combining inequalities \eqref{I},\eqref{II} and \eqref{III}, and choosing suitable $\ep$, we have
\begin{equation}\label{new-H^3-u}
\frac{1}{2}\frac{\p}{\p t}\int_{\Om}|\De u^n|^2dx+\frac{\al}{2}\int_{\Om}|\n \De u^n|^2dx\leq C(\de, \al,\Phi)(U+S+1)^5.
\end{equation}

On the other hand, for equation \eqref{H^3-s}, we have
\begin{equation}\label{IV}
\begin{aligned}
|IV|=&\int_{\Om}\<D_0\De s^n,\De^2 s^n\>dx-\th\int_{\Om}\<\n D_0 \cdot(\J(u^n)\otimes(\n s^n\cdot \J(u^n))),\De^2 s^n\>dx\\
&+\int_{\Om}\<\n D_{0}\cdot \n s^n,\De^2 s^n\>dx-\th\int_{\Om}\<D_0 \n \J(u^n)\otimes(\n s^n\cdot \J(u^n)),\De^2 s^n\>dx\\
&-\th\int_{\Om}\<D_0 \J(u^n)\otimes(\De s^n\cdot \J(u^n)),\De^2 s^n\>dx-\int_{\Om}\<\n u^n\cdot J_e,\De^2 s^n\>dx\\
&-\th\int_{\Om}\<D_0 \J(u^n)\otimes(\n s^n\cdot \n \J(u^n)),\De^2 s^n\>dx-\int_{\Om}\<u^n\mbox{div}J_e,\De^2 s^n\>dx\\
=&IV_1+IV_2+IV_3+IV_4+IV_5+IV_6+IV_7+IV_8.
\end{aligned}
\end{equation}
By direct calculations, there hold
\begin{equation}\label{IV_1}
\begin{aligned}
IV_1\leq &-\int_{\Om}D_0|\n \De s^n|^2dx+C(\ep,\th,c_0)|\n D_0|^2_{L^\infty}S+\ep \th \int_{\Om}D_{0}|\n \De s^n|^2dx,\\
IV_2\leq& C(\ep, \de, c_0)\th(|\n^2 D_{0}|^2_{L^{\infty}}+|\n D_0|^2_{L^\infty})(U+S+SU)+\ep\th\int_{\Om}D_{0}|\n \De s^n|^2dx,\\
IV_3\leq &C(\ep,\th,c_0)(|\n^2 D_{0}|^2_{L^{\infty}}+|\n D_0|^2_{L^\infty})S+\ep \th \int_{\Om}D_{0}|\n \De s^n|^2dx,\\
IV_4\leq &\th \int_{\Om}|\n D_0||\n \J(u^n)||\n s^n||\J(u^n)||\n \De s^n|dx\\
&+\th \int_{\Om}|D_0||\n^2 \J(u^n)||\n s^n||\J(u^n)||\n \De s^n|dx\\
& +\th \int_{\Om}D_0|\n \J(u^n)||\De s^n||\J(u^n)||\n \De s^n|dx\\
&+\th \int_{\Om}D_0|\n \J(u^n)|^2|\n s^n||\n \De s^n|dx\\
\leq& C(\de, \ep,c_0)(|\n D_0|^2_{L^\infty}+|D_0|_{L^\infty})(US^2+US+U^2S)\\
&+\ep\th \norm{\n \De u^n}^2_{L^2}+\ep\th \int_{\Om}D_{0}|\n \De s^n|^2dx,\\
 IV_5\leq&C(\de)\th|\n D_{0}|_{L^\infty}\int_{\Om}|\De s^n||\n \De s^n|dx+\th C(\de)\int_{\Om}D_0|\n u^n||\De s^n||\n \De s^n|dx\\
&+\th \int_{\Om}|\J(u^n)\cdot \n \De s^n|^2dx\\
\leq& (\ep+1+\de) \th\int_{\Om}D_{0}|\n \De s^n|^2dx+C(\de,\ep, c_0)\norm{D_0}_{C^1}\th(S+US+U^2S),\\
 IV_6+IV_8\leq& C(\ep, c_0)\th \norm{J_e}^2_{H^2}U+\ep\th  \int_{\Om}D_{0}|\n \De s^n|^2dx.
\end{aligned}
\end{equation}
In addition, $IV_7$ has the same estimates as $IV_{4}$.

Finally, we show the estimate of term $V$ as follows
 \begin{equation}\label{V}
\begin{aligned}
|V|\leq C(\ep,\th,c_0)\norm{D_0}_{C^1}(S+US)+\ep \th \int_{\Om}D_{0}|\n \De s^n|^2dx.
\end{aligned}
\end{equation}
Substituting inequalities \eqref{IV_1} and \eqref{V} into the formula \eqref{H^3-s}, by choosing suitable $\ep$, there holds
\begin{equation}\label{new-H^3-s}
\begin{aligned}
&\frac{1}{2}\frac{\p}{\p t}\int_{\Om}|\De s^n|^2dx+\frac{1}{2}(1-(1+\de)\th)\int_{\Om}D_0|\n\De s^n|^2dx\\
\leq &C(\de,\th,c_0, \norm{D_0}_{C^2}, \norm{J_e}_{H^2})(U+S+1)^2.
\end{aligned}
\end{equation}

Therefore, by combining \eqref{L^2-us}, \eqref{new-H^3-u} and \eqref{new-H^3-s}, we get
\begin{equation}\label{H^3-us}
\begin{aligned}
&\frac{1}{2}\frac{\p}{\p t}(U+S)+\frac{\al}{2}\int_{\Om}|\n \De u^n|^2dx+\frac{1}{2}(1-(1+\de)\th) \int_{\Om}D_0|\n\De s^n|^2dx\\
\leq &C(\al,\de,\th,c_0, \Phi,\norm{D_0}_{C^2},\norm{J_{e}}_{H^2})(U+S+1)^5.
\end{aligned}
\end{equation}
Here we should point out that $\de$ is a fixed positive number such that $\de<\frac{1}{\th}-1$.

By using the comparison theorem of ODE in \ref{c-thm}, the desired estimates of approximated solution $(u^{n},s^n)$ are obtained from \eqref{H^3-us}. Hence, we conclude
\begin{lem}\label{es-us}
Let $(u_{0},s_0)\in H^{2}(\Om,\Real^3)$. Suppose $J_e\in L^\infty(\Real^+,H^2(\Om,\Real^3))$ and $D_0\in C^2(\bar{\Om})$ with $D_0\geq c_0>0$ for some constant $c_0$. Then there exists a solution $(u^{n},s^n) \in L^{\infty}([0,T^*),H^{2}(\Om,\Real^{3}))\cap W^{3,1}_{2}(\Om\times[0,T^*),\Real^{3})$ to system \eqref{app-llgsp}, where $T^*$ is dependent on  $\norm{u_0}_{H^2(\Om)}+\norm{s_0}_{H^2(\Om)}$. Moreover, for any $T< T^*$, there exists a constant $C(T)$ independent on $(u^{n},s^n)$, such that the following a priori estimate holds
\begin{equation}\label{inq-us}
\begin{aligned}
&\sup_{0< t\leq T}\(\norm{u^n}_{H^2}+\norm{s^n}_{H^2}\)+\al\int_{0}^{T}\norm{\n \De u^{n}}^2_{L^2}dt\\
&\quad+(1-(1+\de)\th)c_0\int_{0}^{T}\norm{\n \De s^{n}}^2_{L^2}dt\leq C(T).
\end{aligned}
\end{equation}
Moreover, the above estimate and equation \eqref{app-llgsp} follows
$$\sup_{0< t\leq T}(\norm{\p_t u^n}^2_{L^2}+\norm{\p_t s^n}^2_{L^2})+\int_{0}^{T}(\norm{\n \p_tu^{n}}^2_{L^2}+\norm{\n \p_ts^{n}}^2_{L^2})dt\leq C(T).$$
\end{lem}
\begin{proof}
Let $y(t)=U(t)+S(t)$, then the estimate \eqref{H^3-us} implies $y$ satisfies the below ODE inequality:
\begin{equation*}
\begin{cases}
y^\prime\leq C(y+1)^5,\\[1ex]
y(0)=\norm{u^n_{0}}^2_{H^2}+\norm{s^n_0}^2_{H^2}\leq C(\norm{u_{0}}^2_{H^2}+\norm{s_0}^2_{H^2}).
\end{cases}
\end{equation*}
Here the constant $C$ is dependent on $\norm{D_0}_{C^2}$ and $\norm{J_e}_{L^\infty(\Real^+,H^2(\Om))}$.
If we let $z: [0,T^*)\to \Real$ be the maximal solution to
\begin{equation*}
\begin{cases}
z^\prime= C(z+1)^5,\\[1ex]
z(0)=C(\norm{u_{0}}^2_{H^2}+\norm{s_0}^2_{H^2}),
\end{cases}
\end{equation*}
where $T^*$ is dependent only on $\norm{u_{0}}^2_{H^2}+\norm{s_0}^2_{H^2}$. Thus, Lemma \ref{c-thm} shows
$$\sup_{0<t\leq T}(U(t)+S(t))\leq C(T).$$
Moreover, by considering estimate \eqref{H^3-us}, it implies \eqref{inq-us}. On the other hand, by using equation \eqref{app-llgsp}, it is not difficult to show the following estimate
$$\sup_{0< t\leq T}(\norm{\p_t u^n}^2_{L^2}+\norm{\p_t s^n}^2_{L^2})+\int_{0}^{T}(\norm{\n \p_tu^{n}}^2_{L^2}+\norm{\n \p_ts^{n}}^2_{L^2})dt\leq C(T).$$
\end{proof}

\subsection{Regular solutions to LLG system with spin-polarized transport}\label{ss: reg-sol}\

In this subsection, we consider the compactness of the approximation solution $(u^{n},s^n)$ to \eqref{app-llgsp} constructed in the above. The main tool to achieve the compactness is the well-known Alaoglu's theorem and the Aubin-Simons's compactness (see Lemma \ref{A-S} in Section \ref{s: pre}). Thus, Lemma \ref{es-us}  implies that there exists a subsequence of $\{(u^{n},s^n)\}$, we still denote it by $\{(u^{n},s^n)\}$, and a $(u,s) \in L^{\infty}([0,T],H^{2}(\Om,\Real^{3}))\cap W^{3,1}_{2}(\Om\times[0,T],\Real^{3})$ such that
$$(u^n,s^n)\rightharpoonup (u,s),\,\,\text{weakly* in}\, L^{\infty}([0,T],H^{2}(\Om,\Real^{3})),$$
$$(u^n,s^n)\rightharpoonup (u,s),\,\,\,\text{weakly in}\,W^{3,1}_{2}(\Om\times[0,T],\Real^{3}),$$
where $0<T<T^*$. Next, let $X=H^{3}(\Om,\Real^{3})$, $B=H^{2}(\Om,\Real^{3})$ and $Y=L^{2}(\Om,\Real^{3})$. Then, Lemma \ref{A-S} tells us
$$(u^n,s^n)\rightarrow (u,s),\,\,\,\text{strongly in}\, L^{p}([0,T], H^{2}(\Om,\Real^{3})),$$
for any $p<\infty$. It follows that $u$ is a strong solution to equation \eqref{llgsp2}.
\begin{theorem}\label{reg-solu}
The limiting map $(u,s)\in L^{\infty}([0,T],H^{2}(\Om,\Real^{3}))\cap W^{3,1}_{2}(\Om\times[0,T],\Real^{3})$ is a locally strong solution to \eqref{llgsp2} for any $0<T<T^*$. Moreover, there exists a constant $C(T)$ such that the following estimate holds
\begin{equation}\label{inq-us1}
\begin{aligned}
&\sup_{0< t\leq T}\(\norm{(u,s)}^2_{H^2}+\norm{(\p_t u,\p_t s)}^2_{L^2}\)+\int_{0}^{T}\norm{(u,s)}^2_{H^3}dt\\
&\quad+\int_{0}^{T}\norm{(\n \p_t u,\n \p_t s)}^2_{L^2}\leq C(T).
\end{aligned}
\end{equation}
\end{theorem}
\begin{proof}
For any $\varphi\in C^{\infty}(\bar{\Om}\times [0,T])$ and $k\in \mathbb{N}$, we set $\varphi_{k}=P_{k}(\varphi)=\sum_{i=1}^{k}g_{i}(t)f_{i}$, where $g_{i}(t)=\<\varphi,f_{i}\>_{L^{2}(\Om)}$. Thus, $\norm{\varphi-\varphi_{k}}_{L^{\infty}([0,T],L^{2}(\Om))}\to 0$ as $k\to \infty$.

We firstly fix $k$ and let $n\geq k$. Since $(u^{n},s^n)$ is a locally strong solution of \eqref{app-llgsp}, it follows	
\begin{align*}
&\int_{0}^{T}\int_{\Om}\<\p_t u^n,\varphi_k\>dxdt-\al \int_{0}^{T}\int_{\Om}\<\De u^n+|\n u^n|u^n,\varphi_k\>dxdt\\
=&-\al \int_{0}^{T}\int_{\Om}\< u^n\times(u^n\times (\tilde{h}+s^n)),\varphi_k\>dxdt-\al \int_{0}^{T}\int_{\Om}\<u^n\times(h+s^n),\varphi_k\>dxdt
\end{align*}
and
\begin{align*}
&\int_{0}^{T}\int_{\Om}\<\p_t s^n,\varphi_k\>dxdt+\int_{0}^{T}\int_{\Om}\<\mbox{div}\(A(\J(u^n)\n s^n)+u^n\times J_e\),\varphi_k\>dxdt\\
=&-\int_{0}^{T}\int_{\Om}\<D_0(x)\cdot s^n+D_0(x)\cdot s^n\times u^n,\varphi_k\>dxdt.
\end{align*}
The above conclusions on compactness tell us that, for any $1\leq p<\infty$,
$$(\p_t u^n,\p_t s^n)\rightharpoonup (\p_tu, \p_t s)\quad\text{weakly in}\, L^{2}([0,T],L^{2}(\Om,\Real^{3})),$$
and
$$(u^n,s^n)\rightarrow (u,s)\quad\text{strongly in}\,\, L^{p}([0,T], H^{2}(\Om,\Real^{3})).$$
It follows
 \begin{itemize}
 	\item $h_d(u^{n})\rightarrow h_d(u)\quad\text{strongly in}\,\, L^{p}([0,T], H^{2}(\Om,\Real^{3}))$,
 	\item $(u^{n},s^n)\rightarrow (u,s)\quad\text{strongly in}\,\, L^p([0,T],L^\infty(\Om,\Real^3))$,
 	\item $u^n\times h(u^n)\to u\times h(u) \quad\text{strongly in}\,\, L^{2}([0,T], L^2(\Om,\Real^{3}))$,
 	\item $u^n\times(u^n\times \tilde{h}(u^n))\to u\times(u\times \tilde{h}(u)) \quad\text{strongly in}\,\, L^{2}([0,T], L^2(\Om,\Real^{3}))$,
 	\item $\mbox{div}(A(u^n)\n s^n)\to\mbox{div}(A(u)\n s) \quad\text{strongly in}\,\, L^{2}([0,T], L^2(\Om,\Real^{3}))$.
 \end{itemize}
Here Lemma \ref{es-h_d} has been used.

Therefore, according to the dominated convergence theorem and the definition of weak convergence, we infer the desired conclusions as $n\to \infty$ and then $k\to \infty$. It remains that we need to check the Neumann boundary condition.

We only check the condition $\frac{\p u}{\p \nu}|_{\p\Om}=0$, the other can be get by a similar argument.
Since for any $\xi\in C^{\infty}(\bar{\Om}\times[0,T])$, there holds
$$\int_{0}^{T}\int_{\Om}\<\De u^{n},\xi\>dxdt=-\int_{0}^{T}\int_{\Om}\<\n u^{n},\n \xi\>dxdt.$$
Letting $n\to \infty$, we have
$$\int_{0}^{T}\int_{\Om}\<\De u,\xi\>dxdt=-\int_{0}^{T}\int_{\Om}\<\n u,\n \xi\>dxdt,$$
this means that $$\frac{\p u}{\p \nu}|_{\p \Om\times [0,T]}=0.$$

Eventually, estimate (\ref{inq-us1}) is obtained from (\ref{inq-us}) by taking $n\to\infty$, since the lower semi-continuity.
\end{proof}

Next, we show $|u|=1$ for any $0<t<T^*$, which follows that $(u,s)$ is a strong solution to \eqref{llgsp1}.
\begin{prop}\label{|u|=1}
The solution $u\in L^{\infty}([0,T],H^{2}(\Om,\Real^{3}))\cap W^{3,1}_{2}(\Om\times[0,T],\Real^{3})$ obtained in the above satisfies $|u|=1$.
\end{prop}
\begin{proof}
Let $\om=|u|^2-1$, which satisfies the following equation
\begin{equation*}
\begin{cases}
\p_t \om-\al \De w=2\al \om|\n u|^2 &(x,t)\in\Om\times(0,T],\\[1,ex]
\om(x,0)=0 & x\in \Om,\\[1ex]
\frac{\p \om}{\p \nu}=0\,&(x,t)\in\p\Om\times[0,T].
\end{cases}
\end{equation*}
If we choose $\om$ as a test function, there holds
$$\frac{1}{2}\frac{\p}{\p t}\int_{\Om}|\om|^2dx\leq 2\al|\n u|^2_{L^\infty}\int_{\Om}|\om|^2dx\leq C\norm{u}^2_{H^3(\Om)}\int_{\Om}|\om|^2dx.$$
Since $\ga(t)=\norm{u}^2_{H^3}(t)$ is in $L^1[0,T]$, the Gronwall inequality implies
$$\int_{\Om}|\om|^{2}dx(t)\leq C(T)\int_{\Om}|\om(x,0)|^2dx=0,\quad 0<t\leq T,$$
for some $C(T)$. It follows $|u|=1$ for almost every $(x,t)\in\Om\times [0,T]$.
\end{proof}

Now, we turn to showing the uniqueness of the solution $(u,s)$ to \eqref{llgsp1} obtained in the above.

\begin{prop}\label{uniqn}
There is a unique solution to \eqref{llgsp1} in $L^{\infty}([0,T],H^{2}(\Om,\Real^{3}))\cap W^{3,1}_{2}(\Om\times[0,T],\Real^{3})$.
\end{prop}
\begin{proof}
Let $(u,s)$ and $(\tilde{u},\tilde{s})$ are two solutions to \eqref{llgsp1} in $L^{\infty}([0,T],H^{2}(\Om,\Real^{3}))\cap W^{3,1}_{2}(\Om\times[0,T],\Real^{3})$ and set $(\bar{u},\bar{s})=(u-\tilde{u},s-\tilde{s})$. Then, $(\bar{u},\bar{s})$ satisfies the following equation
\begin{equation*}
\begin{cases}
\p_t \bar{u}=\De \bar{u}+R_1,\\[1ex]
\p_t \bar{s}=-\mbox{div}(A(u)\n \bar{s})+R_{2},
\end{cases}
\end{equation*}
with initial-boundary condition:
\begin{equation*}
\begin{cases}
u(\cdot,0)=0,\,\,\frac{\p u}{\p \nu}|_{\p \Om}=0,\\[1ex]
s(\cdot,0)=0,\,\,\frac{\p s}{\p \nu}|_{\p \Om}=0.
\end{cases}
\end{equation*}
Here
\begin{align*}
R_1=&|\n u|^2 \bar{u}-(|\n u|^2-|\n \tilde{u}|^2)\tilde{u}+\bar{u}\times(u\times (\tilde{h}(u)+s))\\
&+\tilde{u}\times(\bar{u}\times (\tilde{h}+s))+\tilde{u}\times(\tilde{u}\times (\tilde{h}(u)-\tilde{h}(\tilde{u})+\bar{s}))\\
&+\bar{u}\times(h(u)+s)+\tilde{u}\times(h(u)-h(\tilde{u})+\bar{s}),\\
R_2=& \mbox{div}(\bar{u}\otimes J_e)-\th\(\mbox{div}(D_0\cdot(u\otimes(\n s\cdot u)))-D_0\cdot(\tilde{u}\otimes(\n \tilde{s}\cdot \tilde{u}))\)\\
&-D_0\bar{s}-D_0\bar{s}\times u-D_0\tilde{s}\times \bar{u}.
\end{align*}

By choosing test functions $(\bar{u},\bar{s})$ and $(\De \bar{u},\De \bar{s})$, we take a direct calculation to show
\begin{align*}
&\frac{\p}{\p t}(\norm{\bar{u}}^2_{H^1}+\norm{\bar{s}}^2_{H^1})+\al \int_{\Om}|\De \bar{u}|^2dx+(1-\th)\int_{\Om}D_0|\De \bar{s}|^2dx\\
\leq& C(\al,\th, \norm{D_0}_{C^1},c_0, \norm{J_e}_{H^2})F(t)(\norm{\bar{u}}^2_{H^1}+\norm{\bar{s}}^2_{H^1})
\end{align*}
for a.e $t\in [0,T]$. Here $F(t)=\norm{u}^2_{H^3}+\norm{\tilde{u}}^2_{H^3}+\norm{s}^2_{H^3}+\norm{\tilde{s}}^2_{H^3}+1$, which is in $L^{1}[0,T]$. Thus, the Gronwall inequality implies
$$\norm{\bar{u}}^2_{H^1}+\norm{\bar{s}}^2_{H^1}=0,\quad\text{for any t}\in[0,T].$$
Therefore, the proof is completed.
\end{proof}

\noindent{\bf Proof of Theorem \ref{mth1}}

Immidiately, Theorem \ref{mth1} follows Theorem \ref{reg-solu}, Proposition \ref{|u|=1} and Proposition \ref{uniqn}. Thus, we finish the proof.

\section{Very regular solution}\label{s: very-reg-sol}
However, when $k\geq 3$, the $H^{k+1}$-estimates of solution can't be derived by multiplying the two sides of the Galerkin approximation of \eqref{llgsp2} in the above section by
$$(-1)^k(\De^k u^n,\De^k s^{n}),$$
since the right hand side of equation \eqref{app-llgsp} doesn't satisfy the homogeneous Neumann boundary condition. Thus, to improve the regularity of the solution $(u,s)$ to \eqref{llgsp1} constructed in the previous section, we apply the method used in \cite{CJ}.

 Roughly speaking, under suitable compatibility initial-boundary conditions, which are defined in Section \ref{ss: comp-cond}, we consider equation \eqref{comp-llgsp3} for $k=1$. By using Galerkin approximation method, we can get a solution  $(u_1,s_1)\in L^{\infty}([0,T],H^{2}(\Om,\Real^{3}))\cap W^{3,1}_{2}(\Om\times[0,T],\Real^{3})$, where $0<T<T^*$ and $T^*$ is determined in Theorem \ref{mth1}. Fortunately, a uniqueness argument guarantees that $(u_1,s_1)\equiv(\p_t u,\p_t s)$. Then, we can get $(u,s)\in L^{\infty}([0,T],H^{4}(\Om,\Real^{3}))\cap W^{5,1}_{2}(\Om\times[0,T],\Real^{3})$ by a bootstrap argument using equation \eqref{llgsp1}. Therefore, we can get very higher regularity of $(u,s)$ step by step, by considering the compatible equation \eqref{comp-llgsp3} for $k>1$.

 Next, we divide this section into two parts. In the first part, we give a detailed process to enhance the regularity of $(u,s)$ to $L^{\infty}([0,T],H^{5}(\Om,\Real^{3}))\cap W^{6,1}_{2}(\Om\times[0,T],\Real^{3})$. And then in the other part, the very regular results can be obtained by using method of induction with a similar argument as that in the first part.

\subsection{$H^6$-Regularity of solution}\label{ss: 1-reg}\

Let $(u^n,s^n)$ be the solution of \eqref{app-llgsp} given in the previous section, and let $(u^n_t,s^n_t)=(\p_t u^n,\p_t s^n)$. Then, $(u^n_t,s^n_t)$ satisfies the following equation
\begin{equation}\label{eq-p_t}
\begin{cases}
\p_t u^n_t=\al \De u^n_t+T_1+T_2+T_3+T_4+T_5,\,\,&(x,t)\in \Om\times[0,T^*),\\[1ex]
\p_t u^n_t=-P_n\(\mbox{div}(A(\J(u^n))\n s^n_t+u^n_t\otimes J_e\)+I_1+I_2,\,\,&(x,t)\in \Om\times[0,T^*),
\end{cases}
\end{equation}
with initial condition
\begin{equation*}
\begin{cases}
u^n_t(x,0)=\sum_{i=1}^{n}\p_t g^n_i(0)f_i,\,& x\in\Om,\\[1ex]
s^n_t(x,0)=\sum_{i=1}^{n}\p_t \ga^n_i(0)f_i,\,& x\in\Om.
\end{cases}
\end{equation*}
Here
\begin{equation*}
\begin{cases}
T_1=\al P_n\(u^n\times (u^n\times (h_d(u^n_t)-\p_t \n_u\Phi(\J(u^n)+s^n_t))\),\\[1ex]
T_2=\al P_n\(u^n_t\times (u^n\times (h_d(u^n)-\n_u\Phi(\J(u^n)+s^n))\)\\
\quad\quad+\al P^n\(u^n\times (u^n_t\times (h_d(u^n)-\n_u\Phi(\J(u^n)+s^n))\),\\[1ex]
T_3=-P_n\(u^n\times(\De u^n_t+h_d(u^n_t)-\p_t \n_u\Phi(\J(u^n))+s^n_t)\),\\[1ex]
T_4=-P_n\(u^n_t\times(\De u^n+h_d(u^n)-\n_u\Phi(\J(u^n))+s^n)\),\\[1ex]
T_5=\al P_n\(|\n u^n|^2u^n_t+2\<\n u^n_t,\n u^n\>u^n\),\\[1ex]
I_1=-P_n\(\mbox{div}(\p_t A(\J(u^n))\n s^n+u^n\otimes \p_tJ_e)\),\\[1ex]
I_2=P_n\(D_0s^n_t+D_0s^n_t\times u^n+D_0s^n\times u^n_t\).
\end{cases}
\end{equation*}

For the sake of convenience, we also denote $$I_0=-P_n\(\mbox{div}(A(\J(u^n))\n s^n_t+u^n_t\otimes J_e\).$$
In Lemma \ref{es-us}, we have shown that
$$\norm{u^n_t}^2_{L^2}+\norm{s^n_t}^2_{L^2}=g(t)\leq C(T),\quad\quad 0\leq t\leq T<T^*.$$

Next, in order to show the $H^2$-estimate of $(u^n_t,s^n_t)$, we choose $(-\De u^n_t,-\De s^n_t)$ as a test function for the above equation \eqref{eq-p_t}. Then, there hold true
\begin{align*}
\frac{1}{2}\frac{\p}{\p_t}\int_{\Om}|\n u^n_t|^2dx+\al\int_{\Om}|\De u^n_t|^2dx=-\sum_{i=1}^{5}\int_{\Om}\<T_i,\De u^n_t\>dx,
\end{align*}
and
\begin{align*}
\frac{1}{2}\frac{\p}{\p_t}\int_{\Om}|\n s^n_t|^2dx=-\sum_{i=0}^{3}\<I_i, \De s^n_t\>dx.
\end{align*}
Next, we estimate the terms on the right hand sides of the above equations step by step.
\begin{align*}
\left|\int_{\Om}\<T_1, \De u^n_t\>dx\right|\leq& C(\ep,\al,\de)U^2(\norm{u^n_t}^2_{L^2}+\norm{s^n_t}^2_{L^2})+\ep \al\norm{\De u^n_t}^2_{L^2},\\
\left|\int_{\Om}\<T_2, \De u^n_t\>dx\right|\leq& C(\ep,\al,\de)U(U+S+1)\norm{u^n_t}^2_{L^2}+\ep \al\norm{\De u^n_t}^2_{L^2},\\
\left|\int_{\Om}\<T_3, \De u^n_t\>dx\right|\leq& C(\ep,\al,\de)U(\norm{u^n_t}^2_{L^2}+\norm{s^n_t}^2_{L^2})+\ep \al\norm{\De u^n_t}^2_{L^2},\\
\left|\int_{\Om}\<T_4, \De u^n_t\>dx\right|\leq& C(\ep,\al,\de)(\norm{u^n}^2_{H^3}+S+1)\norm{u^n_t}^2_{H^1}+\ep \al\norm{\De u^n_t}^2_{L^2},\\
\left|\int_{\Om}\<T_5, \De u^n_t\>dx\right|\leq &C(\ep,\al,\de)U(U+\norm{u^n}^2_{H^3})\norm{u^n_t}^2_{H^1}+\ep \al\norm{\De u^n_t}^2_{L^2},
\end{align*}

\begin{align*}
-\int_{\Om}\<I_0, \De s^n_t\>dx \leq& C(\ep,\al,\de, \norm{D_0}_{C^1})(U+\norm{J_e}^2_{H^2}+1)(\norm{u^n_t}^2_{H^1}+\norm{\n s^n_t}^2_{L^2})\\
&-(1-(1+\de+\ep)\th)\int_{\Om}D_0|\De s^n_t|^2dx,\\
\left|\int_{\Om}\<I_1, \De s^n_t\>dx\right|\leq& C(\ep,\th,\de,\norm{D_0}_{C^1})(\norm{s^n}^2_{H^3}+S\norm{u^n}^2_{H^3})\norm{u^n_t}^2_{H^1}\\
&+C U\norm{\p_t J_e}^2_{H^1}+\ep \th\int_{\Om}D_0|\De s^n_t|^2dx,\\
\left|\int_{\Om}\<I_2, \De s^n_1\>dx\right|\leq& C(\ep,\th,c_0)(U+S+1)(\norm{u^n_t}^2_{L^2}+\norm{s^n_t}^2_{L^2})+\ep \th\int_{\Om}D_0|\De s^n_t|^2dx.
\end{align*}
By combining the above inequalities, we obtain
\begin{equation}\label{es-u_ts_t}
\begin{aligned}
&\frac{\p}{\p t}(\int_{\Om}|\n u^n_t|^2dx+\int_{\Om}|\n s^n_t|^2dx)+\al\int_{\Om}|\De u^n_t|^2dx+(1-(1+\de)\th)\int_{\Om}D_0|\De s^n_t|^2dx\\
\leq &f(t)(\norm{u^n_t}^2_{L^2}+\norm{s^n_t}^2_{L^2}+1)+f(t)(\int_{\Om}|\n u^n_t|^2dx+\int_{\Om}|\n s^n_t|^2dx).
\end{aligned}
\end{equation}
Here
$$f(t)=C(\al,\de,\th, \norm{D_0}_{C^1})(U^2+S^2+(S+1)\norm{u^n}^2_{H^3}+\norm{J_e}^2_{H^2}+U\norm{\p_t J_e}^2_{H^1}+1),$$
which is in $L^1([0,T])$ for any $0<T<T^*$. Here we have used that $J_e\in L^2([0,T],H^2(\Om))$ and $\p_t J_e\in L^2([0,T],H^1(\Om))$.

Thus, the Gronwall inequality and (\ref{es-u_ts_t}) imply
\begin{equation}\label{est-u_ts_t}
\begin{aligned}
&\sup_{0<\leq t\leq T}(\int_{\Om}|\n u^n_t|^2dx+\int_{\Om}|\n s^n_t|^2dx)+\al\int_{0}^{T}\int_{\Om}|\De u^n_t|^2dx\\
&+(1-(1+\de)\th)\int_{0}^{T}\int_{\Om}D_0|\De s^n_1|^2dx\\
\leq& C(T, \norm{\n u^n_t}^2_{L^2}(0)+\norm{\n s^n_t}^2_{L^2}(0), \norm{f}_{L^1([0,T])}).
\end{aligned}
\end{equation}
On the other hand, equation \eqref{app-llgsp} implies
\begin{equation}
\begin{cases}
u^n_t(x,0)= P_n\(\al(\De u^n_0+|\n u^n_0|^2u^n_0)-u^n_0\times (u^n_0\times(\tilde{h}+s^n_0)))-u^n_0\times(h+s^n_0)\),\\[1ex]
u^n_t(x,0)=-P_n\(\mbox{div}\(A(\J(u^n_0))\n s^n_0+u^n_0\times J_e\)-D_0(x)\cdot s^n_0-D_0(x)\cdot s^n_0\times u^n_0\).
\end{cases}
\end{equation}
A direct computation shows
\begin{align*}
\norm{\n u^n_t}^2_{L^2}(0)+\norm{\n s^n_t}^2_{L^2}(0)\leq& C (\norm{u^n_0}_{H^3(\Om)}+\norm{s^n_0}_{H^3(\Om)}+1)\\
\leq &C(\norm{u_0}_{H^3(\Om)}+\norm{s_0}_{H^3(\Om)}+1),
\end{align*}
since we have the following estimates
$$\norm{u^n_0}_{H^3(\Om)}\leq C\norm{u_0}_{H^3(\Om)},$$
$$\norm{s^n_0}_{H^3(\Om)}\leq C\norm{s_0}_{H^3(\Om)},$$
and $J_e\in C^0(\Real^+,H^2(\Om))$.

Now, we can give the proof of Theorem \ref{mth2} by using the above estimate of $(u^n_t,s^n_t)$ and a bootstrap argument.

\begin{proof}[\textbf{The proof of Theorem \ref{mth2}}]
The estimates in (\ref{est-u_ts_t}) indicates
$$(u^n_t,s^n_t)\rightharpoonup (\p_tu, \p_t s)\quad\text{weakly in}\, L^{2}([0,T],H^{2}(\Om,\Real^{3})),$$
and
$$(u^n_t,s^n_t)\rightharpoonup (\p_tu, \p_t s)\quad\text{weakly* in}\, L^{\infty}([0,T],H^{1}(\Om,\Real^{3})).$$
Thus,
$$(\p_tu, \p_t s)\in L^\infty([0,T],H^1(\Om,\Real^3))\cap L^2([0,T],H^2(\Om,\Real^3)).$$

Next, by using the property of cross-product, the equation \eqref{llgsp} can be rewritten as the form
\begin{equation}\label{llgsp'}
\begin{cases}
\De u=-|\n u|^2u+u\times (u\times(\tilde{h}+s))-\frac{1}{1+\al^2}\(\al \p_t u+u\times \p_t u\),&\text{(x,t)}\in\Om\times \Real^+,\\[1ex]
L s=\mbox{div}(A(u)\n s)=-\p_t s+\mbox{div}(u\otimes J_e)+D_0(x)\cdot s+D_0(x)\cdot s\times u, &\text{(x,t)}\in\Om_0\times \Real^+.
\end{cases}
\end{equation}

\noindent\emph{We claim: $\De u$ and $\De s\in L^\infty([0,T],H^1(\Om,\Real^3))\cap L^2([0,T],H^2(\Om,\Real^3)).$} \medskip

Here, we shall only prove $\De u\in L^\infty([0,T],H^1(\Om,\Real^3))\cap L^2([0,T],H^2(\Om,\Real^3))$, since we can take a similar argument to prove the fact $\De s$ belongs to the same space as $\De u$.

From the first equation of \eqref{llgsp'} and $u\in L^\infty([0,T],H^2(\Om,\Real^3))\cap L^2([0,T],H^3(\Om,\Real^3))$, we can easily get that
$$\sup_{0<t\leq T}\norm{\De u}_{L^3}\leq C(T).$$
Since
\begin{align*}
\n \De u=&-(|\n u|^2\n u+2\n^2 u\cdot \n u u)+\n \(u\times (u\times(\tilde{h}+s))\)\\
&+\frac{1}{1+\al^2}(\al \n \p_t u+\n u\times \p_tu+u\times \n \p_t u),
\end{align*}
by a simple calculation we obtain
$$\norm{\n \De u}^2_{L^2}\leq C(U+S)^3+U\norm{\n^2 u}^2_{L^3}+(U+1)\norm{\p_t u}^2_{H^1}\leq C(T),$$
where we have used the global $L^p$ estimate (see Theorem 2.3 in \cite{Weh})
$$\norm{\n^2 u}_{L^3}\leq C(\norm{\De u}_{L^3}+\norm{u}_{L^3}).$$
Thus, we have $u\in L^\infty([0,T],H^3(\Om,\Real^3))$ by classical $L^p$-estimate.

In the sense of distribution, we have
\begin{align*}
\n^2 \De u=F=&-\n(|\n u|^2\n u+2\n^2 u\cdot \n u u)+\n^2 \(u\times (u\times(\tilde{h}+s))\)\\
&+\frac{1}{1+\al^2}\n(\al \n \p_t u+\n u\times \p_tu+u\times \n \p_t u).
\end{align*}
By a simple calculation, we get
$$\norm{F}^2_{L^2}\leq C(U+S+1)^2+(U+1)\norm{ u}^4_{H^3}+\norm{\p_t u}^2_{H^{1}}\norm{ u}^2_{H^3}+\norm{\n^2 \p_tu}^2_{L^2},$$
which implies that
$$\De u\in L^{2}([0,T],H^2(\Om,\Real^3)).$$
Thus, we have
 $$u\in L^2([0,T],H^4(\Om,\Real^3)).$$
Therefore, we have finished the proof of the claim on $\De u$.

Now we turn to considering the regularity of $\De s$. Since we have shown that
$$u\in L^\infty([0,T],H^3(\Om,\Real^3))\cap L^2([0,T],H^4(\Om,\Real^3))$$
and $L s=\mbox{div}(A(u)\n s)$ is a uniformly elliptic operator, a similar argument to that in the above implies
 $$\De s \in L^\infty([0,T],H^1(\Om,\Real^3))\cap L^2([0,T],H^2(\Om,\Real^3)),$$
Here we need $D_0\in C^3(\bar{\Om})$. Thus, it follows $s \in L^\infty([0,T],H^3(\Om,\Real^3))\cap L^2([0,T],H^4(\Om,\Real^3))$.
\end{proof}

\begin{rem}
We can also show $(u,s)\in C^0([0,T], H^3(\Om,\Real^3))$ by the embedding Theorem $II.5.14$ in \cite{BF}.
\end{rem}

\begin{rem}
We cannot get  $H^3$-estimate of $(u_t,s_t)$ as that for $(u,s)$ in the previous section by applying test function $(\De^2 u^n_t, \De^2 s^n_t)$, since the following inequality does not hold true
$$\norm{u^n_0}_{H^4}\leq C\norm{u_0}_{H^4},$$
which is beyond of our knowledge. Meanwhile, the same reason implies that we  also can't get higher regularity of $(u,s)$ by considering the equation of $(\p^2_{t} u^n, \p^2_{t} s^n)$.
\end{rem}

The above remark indicates that one can not proceed to enhance the regularity except for one adds some new restrictive conditions on the initial data. Motivated by \cite{CJ}, we intend to adding some suitable compatibility boundary conditions to enhance the regularity (see \cite{CJ}). Hence, we consider the following equation
\begin{equation}\label{comp-llgsp}
\begin{cases}
\p_t u_1=\al \De u_1-u\times \De u_1+K_1(\n u_1,\n s_1)+L_1(s_1,u_1),\quad\quad&\text{(x,t)}\in\Om\times \Real^+,\\[1ex]
\p_t s_1=-\mbox{div}\(A(u)\n s_1\)+\hat{Q}_1(\n u_1,\n s_1)+T_1(u_1,s_1), \quad\quad&\text{(x,t)}\in\Om\times \Real^+,\\[1ex]
\frac{\p u_1}{\p \nu}=0,\quad\quad&\text{(x,t)}\in\p \Om\times \Real^+,\\[1ex]
\frac{\p s_1}{\p \nu}=0,\quad\quad&\text{(x,t)}\in\p \Om\times \Real^+.
\end{cases}
\end{equation}
with initial-boundary condition
\begin{equation}
\begin{cases}
u_1(x,\cdot)=V_1(u_0,s_0),\quad \frac{\p V_1}{\p \nu}|_{\p \Om}=0;\\[1ex]
s_1(x,\cdot)=W_1(u_0,s_0),\quad \frac{\p W_1}{\p \nu}|_{\p \Om}=0.
\end{cases}
\end{equation}
Here,
\begin{align*}
K_1(\n u_1,\n s_1)=&2\al (\n u_1\cdot \n u) u,\\
\hat{Q}_1(\n u_1,\n s_1)=&-\mbox{div}(u_1\otimes J_e)-\mbox{div}(u\otimes \p_tJ_e),\\
L_1(u_1,s_1)=&\al |\n u|^2u_1-\al u_1\times (u\times(\tilde{h}(u)+s))-\al u\times (u_1\times(\tilde{h}(u)+s))\\
&-\al u_1\times (u\times(\bar{h}(u_1)+s_1))-u\times(\bar{h}(u_1)+s_1)-u_1\times(h(u)+s),\\
T_1(u_1,s_1)=&\th \mbox{div}(D_0u_1\otimes(\n s\cdot u+D_0u\otimes(\n s\cdot u_1)))-D_0s_1\\
&-D_0s_1\times u-D_0 s\times u_1,\\
V_1(u_0,s_0)=&\al\(\De u_0+|\n u_0|^2u_0-u_0\times (u_0\times(\tilde{h}(u_0)+s_0))\)\\
&-u_0\times(\De u_0 +\tilde{h}(u_0)+s_0),\\
W_1(u_0,s_0)=&-\mbox{div}\(A(u_0)\n s_0+u_0\otimes J_e(x,0)\)-D_0(x)\cdot s_0-D_0(x)\cdot s_0\times u_0,
\end{align*}
where $\bar{h}(u_1)=h_d(u_1)-\n^2\Phi(u)\cdot u_1$.

The above equation \eqref{comp-llgsp} is a linear system with respect to $u_1$ and $s_1$, since $h_d$ is a linear operator. Next, as before, we also need to consider the following Galerkin approximation equation associated to \eqref{comp-llgsp}
\begin{equation}\label{app-comp-llgsp}
\begin{cases}
\p_t u^n_1=P_n\(\al\De u^n_1-u\times \De u^n_1+K_1(\n u^n_1,\n s^n_1)+L_1(s^n_1,u^n_1)\),\quad\quad&\text{(x,t)}\in\Om\times \Real^+,\\[1ex]
\p_t s^n_1=P_n\(-\mbox{div}(A(u)\n s^n_1)+\hat{Q}_1(\n u^n_1,\n s^n_1)+T_1(u^n_1,s^n_1)\), \quad\quad&\text{(x,t)}\in\Om\times \Real^+\\[1ex]
\end{cases}
\end{equation}
with initial data
\begin{equation}
\begin{cases}
u^n_1(x,\cdot)=P_n(V_1(u_0,s_0)),\\[1ex]
s^n_1(x,\cdot)=P_n(W_1(u_0,s_0)).
\end{cases}
\end{equation}

By the almost same argument as that on equation \eqref{app-llgsp} in Section \ref{s: reg-sol}, it is not difficult to prove that there exists a solution $(u^n_1,s^n_1)$ in $L^\infty([0,T^*),H^2(\Om,\Real^3))\cap L^2([0,T^*),H^3(\Om,\Real^3))$ with energy estimates as follows
\begin{equation}\label{newinq-us1}
\begin{aligned}
&\sup_{0< t\leq T}\(\norm{u^n_1}_{H^2(\Om)}+\norm{s^n_1}_{H^2(\Om)}\)+\al\int_{0}^{T}\norm{\n \De u^{n}_1}^2_{L^2(\Om)}dt+(1-\th)c_0\int_{0}^{T}\norm{\n \De s^{n}_1}^2_{L^2(\Om)}dt\\
\leq& C(T, \norm{(u_0,s_0)}_{H^2(\Om)},\norm{(V_{1},W_1)}_{H^2(\Om)}),
\end{aligned}
\end{equation}
and
\begin{equation}\label{newinq-us2}
\begin{aligned}
&\sup_{0< t\leq T}(\norm{\p_t u^n_1}^2_{L^2(\Om)}+\norm{\p_t s^n_1}^2_{L^2(\Om)})+\int_{0}^{T}(\norm{\n \p_tu^{n}_1}^2_{L^2(\Om)}+\norm{\n \p_ts^{n}_1}^2_{L^2(\Om)})dt\\
\leq& C(T, \norm{(u_0,s_0)}_{H^2(\Om)},\norm{(V_{1},W_1)}_{H^2(\Om)}),
\end{aligned}
\end{equation}
for any $0<T<T^*$, since the system \eqref{comp-llgsp} is linear. Here we need assume
$$J_e\in C^0([0,T],H^3(\Om,\Real^3))\quad\text{and}\quad \p_tJ_e\in L^\infty([0,T],H^2(\Om,\Real^3)),$$
and use the estimate:
\begin{equation*}
\begin{aligned}
\norm{(P_n(V_1),P_n(W_1))}_{H^2(\Om)}\leq C \norm{(V_1,W_1)}_{H^2(\Om)}\leq C(\norm{(u_0,s_0)}_{H^4(\Om)}).
\end{aligned}
\end{equation*}

By the Alaoglu's theorem and Aubin-Simon compactness Theorem \ref{A-S}, we know that there exists a subsequence of $\{(u^{n}_1,s^n_1)\}$, we still denote it by $\{(u^{n}_1,s^n_1)\}$, and a map $$(u_1,s_1) \in L^{\infty}([0,T],H^{2}(\Om,\Real^{3}))\cap W^{3,1}_{2}(\Om\times[0,T],\Real^{3})$$ such that, for any $1<p<\infty$,
\begin{itemize}
	\item $(u^n_1,s^n_1)\rightharpoonup (u_1,s_1),\,\,\text{weakly* in}\, L^{\infty}([0,T],H^{2}(\Om,\Real^{3}))$,
	\item $(u^n_1,s^n_1)\rightharpoonup (u_1,s_1),\,\,\,\text{weakly in}\,W^{3,1}_{2}(\Om\times[0,T],\Real^{3})$,
	\item $(u^n_1,s^n_1)\to (u_1,s_1),\,\,\,\,\,\text{strong in}\,L^p([0,T],H^2(\Om, \Real^3))$.
\end{itemize}
It follows that $(u_1, s_1)$ is a strong solution to \eqref{comp-llgsp}. Meanwhile, it is not difficult to show $(u_t,s_t)$ is a solution to \eqref{comp-llgsp} in $L^\infty([0,T],H^1(\Om,\Real^3))\cap L^2([0,T],H^2(\Om,\Real^3))$, since $(u,s)\in C^0([0,T],H^3(\Om,\Real^3))$.  \medskip

If the uniqueness of solution to \eqref{comp-llgsp} in $L^\infty([0,T],H^1(\Om,\Real^3))\cap L^2([0,T],H^2(\Om,\Real^3))$ can be established well, then one can see easily that $(u_t,s_t)\equiv(u_1,s_1)$, which implies higher regularity. So, we need to prove the following

\begin{prop}\label{uniq}
There is a unique solution to \eqref{comp-llgsp} in $L^\infty([0,T],H^1(\Om,\Real^3))\cap L^2([0,T],H^2(\Om,\Real^3))$.
\end{prop}
\begin{proof}
Assume that $(u_1,s_1)$ and $(\tilde{u}_1,\tilde{s}_1)$ are two solutions to \eqref{comp-llgsp} in $L^{\infty}([0,T^*),H^{1}(\Om,\Real^{3}))\cap W^{2,1}_{2}(\Om\times[0,T^*),\Real^{3})$. Let$(\bar{u}_1,\bar{s}_1)=(u_1-\tilde{u}_1,s_1-\tilde{s}_1)$, it satisfies the following equation
\begin{equation*}
\begin{cases}
\p_t \bar{u}_1=\al \De \bar{u}_1-u\times \De \bar{u}_1+K_1(\n \bar{u}_1,\n \bar{s}_1)+L_1(\bar{u}_1,\bar{s}_1),\\[1ex]
\p_t \bar{s}_1=-\mbox{div}(A(u)\n \bar{s}_1)+Q_1(\n \bar{u}_1,\n \bar{s}_1)+T_1(\bar{u}_1,\bar{s}_1),
\end{cases}
\end{equation*}
with initial-boundary condition:
\begin{equation*}
\begin{cases}
\bar{u}_1(\cdot,0)=0,\,\,\frac{\p \bar{u}_1}{\p \nu}|_{\p \Om}=0,\\[1ex]
\bar{s}_1(\cdot,0)=0,\,\,\frac{\p \bar{s}_1}{\p \nu}|_{\p \Om}=0.
\end{cases}
\end{equation*}
The above equation indicates $(\p_t\bar{u}_1,\p_t\bar{s}_1)\in L^2([0,T]\times\Om)$. Then by taking $(\bar{u}_1,\bar{s}_1)$ as a test function of the above equation, there holds
\begin{align*}
&\frac{1}{2}\frac{\p }{\p t}\int_{\Om}(|\bar{u}_1|^2+|\bar{s}_1|^2)dx\\
\leq &C(\al, \th, D_0, \norm{\p_tJ_e}_{H^2}+\norm{J_e}_{H^2})(\norm{(u,s)}^2_{H^3}+U+S+1)\int_{\Om}(|\bar{u}_1|^2+|\bar{s}_1|^2)dx\\
\leq &F(t)\int_{\Om}(|\bar{u}_1|^2+|\bar{s}_1|^2)dx,
\end{align*}
where $F(t)\in L^1([0,T])$ for any $0<T<T^*$. Hence, the Gronwall inequality implies $(\bar{u}_1,\bar{s}_1)=(0,0)$ and the proof is finished.
\end{proof}
Proposition \ref{uniq} tells us that $(u_t,s_t)\in L^\infty([0,T],H^2(\Om,\Real^3))\cap L^2([0,T],H^3(\Om,\Real^3))$. If $(u_0,s_0)\in H^5(\Om,\Real^3)$ and $D_0\in C^5(\bar{\Om})$, then a similar argument with that in the proof of Theorem \ref{mth2} shows
$$(u,s)\in L^\infty([0,T],H^5(\Om,\Real^3))\cap L^2([0,T],H^6(\Om,\Real^3)).$$

\subsection{Very regular solutions}\label{ss: main-thm}\

In this subsection, we will apply the method of induction to show the existence of very regular solution $(u,s)$ to \eqref{llgsp1} by considering the initial-Neumann boundary value problem of the equation of $(u_k=\p^k_t u,s_k=\p^k_t s)$ with matching initial-boundary data, that is to prove Theorem \ref{mth3}. In fact, in the above subsection \ref{ss: 1-reg}, we have enhanced regularity of $(u,s)$ by the strategy $\P$ in Section \ref{s: intro}.

For the case of $k>1$, to prove Theorem \ref{mth3} we need to repeat the process of the strategy $\P$ by showing the following property $\P(k)$:
\begin{enumerate}
\item If $(u_0,s_0)\in H^{2k}(\Om,\Real^3)$ with compatibility condition $k-1$ order, then the solution $(u,s)$ with initial date $(u_0,s_0)$ satisfies
$$(\p^i_t u,\p^i_t s)\in L^\infty([0,T], H^{2k-2i}(\Om,\Real^3))\cap L^2([0,T], H^{2k-2i+1}(\Om,\Real^3)), \quad 0<T<T^*,$$
where $i\in \{0,\dots,k\}$;
\item Moreover, if $(u_0,s_0)\in H^{2k+1}(\Om,\Real^3)$, then there holds
$$(\p^i_t u,\p^i_t s)\in L^\infty([0,T], H^{2k-2i+1}(\Om,\Real^3))\cap L^2([0,T], H^{2k-2i+2}(\Om,\Real^3)), \quad 0<T<T^*,$$
where $i\in \{0,\dots,k\}$.
\end{enumerate}

In fact, it is not difficult to see that, in the previous subsection, we have shown this property holds for $k=1$ and $k=2$.

Next, suppose that $\P(k)$ is already established for $k\geq 2$, then we want to prove that $\P(k+1)$ is true. Therefore, we assume $(u_0,s_0)\in H^{2(k+1)}(\Om,\Real^3)$ satisfies the compatibility condition  \eqref{com-cond} at $k$ order. By the property $\P(k)$, we have
   \[(\p^i_t u,\p^i_t s)\in L^\infty([0,T], H^{2k-2i+1}(\Om,\Real^3))\cap L^2([0,T], H^{2k-2i+2}(\Om,\Real^3)),\]
where $0<T<T^*$ and $ i\in \{0,\dots,k\}$.

Furthermore, we know that $(u_k,s_k)\in L^\infty([0,T],H^1(\Om,\Real^3))\cap L^2([0,T],H^2(\Om,\Real^3))$ satisfies the following equation
\begin{equation}\label{comp-llgsp2}
\begin{cases}
\p_t v=\al \De v-u\times \De v+K_k(\n v,\n w)+L_k(v,w)+F_k(u,s),&\text{(x,t)}\in\Om\times \Real^+,\\[1ex]
\p_t w=-\mbox{div}\(A(u)\n w\)+Q_k(\n v,\n w)+T_k(v,w)+Z_k(u,s),&\text{(x,t)}\in\Om\times \Real^+,\\[1ex]
\frac{\p v}{\p \nu}=0,\quad\quad&\text{(x,t)}\in\p \Om\times \Real^+,\\[1ex]
\frac{\p w}{\p \nu}=0,\quad\quad&\text{(x,t)}\in\p \Om\times \Real^+,\\[1ex]
(v(x,0),w(x,0))=(V_k(u_0,s_0),W_k(u_0,s_0)),\quad&x\in\Om.
\end{cases}
\end{equation}

Next, we will adopt the same procedure as in the strategy $\P$ for $(\p^{k}_tu, \p^{k}_ts)$ to get the regular property $\P(k+1)$. Hence, we need to show the following three claims.
\begin{enumerate}
\item If $(u_0,s_0)\in H^{2k+2}(\Om,\Real^3)$,  then we can get a regular solution to \eqref{comp-llgsp2}
$$(v,w)\in L^\infty([0,T],H^2(\Om,\Real^3))\cap L^2([0,T],H^3(\Om,\Real^3))$$
with the $k$-order compatibility condition of initial date at boundary.
\item There holds true $(\p^k_t u,\p^k_t s)\equiv(v,w)$ as long as one can show the uniqueness of solution to \eqref{comp-llgsp2} in the space $L^\infty([0,T],H^1(\Om,\Real^3))\cap L^2([0,T],H^2(\Om,\Real^3))$. Moreover, It implies $$(u,s)\in  L^\infty([0,T],H^{2k+2}(\Om,\Real^3))\cap L^2([0,T],H^{2k+3}(\Om,\Real^3))$$ by using equation \eqref{comp-llgsp3} again.
\item If $(u_0,s_0)\in H^{2k+3}(\Om,\Real^3)$, then, one can infer that $$(u,s)\in  L^\infty([0,T],H^{2k+3}(\Om,\Real^3))\cap L^2([0,T],H^{2k+4}(\Om,\Real^3))$$ by repeating the same arguments as one proves item $(1)$ of the property $\P$ on $(\p^k_t u,\p^k_t s)$.
\end{enumerate}

In the below context, we will show the above three claims step by step.
\subsubsection{\textbf{Regular solution to \eqref{comp-llgsp2}}}
 Now, we repeat the process of Galerkin approximation in Section \ref{s: reg-sol} to seek a regular solution $(v,w)\in L^\infty([0,T], H^2(\Om,\Real^3))\cap L^2([0,T], H^3(\Om,\Real^3))$ to equation \eqref{comp-llgsp2} when $(V_k,W_k)\in H^2(\Om,\Real^3)$. By a similar argument of enhancing regularity to that in the previous subsection \ref{ss: 1-reg}, we can obtain that $(v,w)\in L^\infty([0,T], H^3(\Om,\Real^3))\cap L^2([0,T],H^4(\Om,\Real^3))$ under an improved initial value assumption that $(V_k, W_k)\in H^3(\Om,\Real^3)$. To this end, first we need to show the estimates of the nonhomogeneous terms $F_k$ and $Z_k$ in \eqref{comp-llgsp2} satisfy the following
 \begin{prop}\label{es-F_kZ_k}
 Assume that the property $\P(k)$ has been established. Suppose that for any $0\leq i\leq k$,
 $$\p^i_tJ_e\in L^\infty([0,T], H^{2k-2i+2}(\Om,\Real^3))\cap L^2([0,T], H^{2k-2i+3}(\Om,\Real^3)).$$
 Then there hold
 $$(F_i,Z_i)\in L^\infty([0,T], H^{2k-2i+1}(\Om,\Real^3))\cap  L^2([0,T], H^{2k-2i+2}(\Om,\Real^3)),$$
 where $0<T<T^*$.
\end{prop}
\begin{proof}
Firstly, we show
$$F_i\in L^\infty([0,T], H^{2k-2i+1}(\Om,\Real^3))\cap  L^2([0,T], H^{2k-2i+2}(\Om,\Real^3)).$$
According to the definition of $F_{i}$ in \eqref{comp-llgsp3}, we only need to consider the following terms.
\begin{enumerate}
	\item[$(a)$] $I=\n u_s\#\n u_j\#u_l$, where $s+j+l=i$ and $0\leq s,j,l\leq i-1$. For $l\in \{1,\cdots,i-1\}$,
	$$u_l\in L^\infty([0,T], H^{2k-2i+3}(\Om,\Real^3))\cap  L^2([0,T], H^{2k-2i+4}(\Om,\Real^3)).$$
	Thus $\n u_s$, $\n u_j$ and $u_l$ are all in $L^\infty([0,T], H^{2k-2i+2}(\Om,\Real^3))$. Since $2k-2i+2\geq 2$, by Lemma \ref{algebra2} we have
	$$I\in L^\infty([0,T], H^{2k-2i+2}(\Om,\Real^3)).$$
	
	\item[$(b)$] $II=\n u_s\#\n u_j\#(\bar{h}(u_l)+s_l)$, where $s+j+l=i$ and $0\leq s,j,l\leq i-1$. The almost same argument as that in $(a)$ shows
	$$II\in L^\infty([0,T], H^{2k-2i+2}(\Om,\Real^3)).$$
	
	\item[$(c)$] $III= \n u_j\#(\bar{h}(u_l)+s_l)$, where $j+l=i$ and $0\leq j,l\leq i-1$. The almost same argument as that in $(a)$ shows
	$$III\in L^\infty([0,T], H^{2k-2i+2}(\Om,\Real^3)).$$
	
	\item[$(d)$] $IV=u_l\#\De u_j$, where $j+l=i$ and $0\leq j,l\leq i-1$. Since
	$$u_l\in L^\infty([0,T], H^{2k-2i+3}(\Om,\Real^3))\cap L^2([0,T], H^{2k-2i+4}(\Om,\Real^3))$$ and
	$$\De u_j\in L^\infty([0,T], H^{2k-2i+1}(\Om,\Real^3))\cap L^2([0,T], H^{2k-2i+2}(\Om,\Real^3)),$$
	Lemma \ref{algebra2} tells us
	$$IV\in L^\infty([0,T], H^{2k-2i+1}(\Om,\Real^3))\cap  L^2([0,T], H^{2k-2i+2}(\Om,\Real^3)).$$
	
	\item[$(e)$] $V=R_i$. By almost the same argument as that in $(a)$ we also have
	$$V\in L^\infty([0,T], H^{2k-2i+2}(\Om,\Real^3)).$$
\end{enumerate}

Next, we turn to the estimates of $Z_i$ and take a similar discussion to that we derive the estimate of $F_i$. It is not difficult to find that one need only to consider the following two terms.
\begin{enumerate}
	\item[$(f)$] $I^{\prime}=u_s\#\n^2 s_j\#u_l$, where $s+j+l=i$ and $0\leq s,j,l\leq i-1$. For $l\in \{1,\cdots,i-1\}$, it is easy to see that $u_s$ and $u_l$ are in $L^\infty([0,T], H^{2k-2i+3}(\Om,\Real^3))$. Then, by Lemma \ref{algebra2} we have
	$$u_s\ast u_l\in L^\infty([0,T], H^{2k-2i+3}(\Om,\Real^3)).$$
Thus, combining the fact and the following
	$$\n^2 s_j\in L^\infty([0,T], H^{2k-2i+1}(\Om,\Real^3))\cap L^2([0,T], H^{2k-2i+2}(\Om,\Real^3)),$$
we have
	$$I^\prime\in L^\infty([0,T], H^{2k-2i+1}(\Om,\Real^3))\cap  L^2([0,T], H^{2k-2i+2}(\Om,\Real^3)),$$
	where we have used Lemma \ref{algebra1}.
	
	\item[$(g)$] $II^\prime=\n u_s\#\p^j_t J_e+u_s\#\n \p^j_t J_e$, where $s+j=i$ and $s<i$. Since there hold that
	$$u_s\in L^\infty([0,T], H^{2k-2i+3}(\Om,\Real^3))\cap  L^2([0,T], H^{2k-2i+4}(\Om,\Real^3))$$
	and
	$$\p^j_tJ_e\in L^\infty([0,T], H^{2k-2i+2}(\Om,\Real^3))\cap L^2([0,T], H^{2k-2i+3}(\Om,\Real^3)),$$
	it is easy to conclude that there holds
	$$II^\prime\in L^\infty([0,T], H^{2k-2i+1}(\Om,\Real^3))\cap  L^2([0,T], H^{2k-2i+2}(\Om,\Real^3)).$$
	\end{enumerate}
\end{proof}
\begin{rem}
	Moreover, if
	$$\p_t^{k+1} J_e\in L^\infty([0,T], L^{2}(\Om,\Real^3))\cap L^2([0,T], H^{1}(\Om,\Real^3)),$$
	by almost the same argument as that in above Proposition \ref{es-F_kZ_k}, we obtain
	$$(\p_t F_i,\p_t Z_{i})\in L^{2}([0,T],H^{2k-2i}(\Om,\Real^3)),$$
	where $i\in\{0,\dots,k\}$.
\end{rem}

Now, we turn to considering the Galerkin approximation associated to \eqref{comp-llgsp2} as following
 \begin{equation}\label{app-comp-llgsp1}
 \begin{cases}
 \p_t v^n=P_n\(\al\De v^n-u\times \De v^n+K_k(\n v^n,\n w^n)+L_k(v^n,w^n)+F_k\)&\text{(x,t)}\in\Om\times \Real^+,\\[1ex]
 \p_t w^n=P_n\(-\mbox{div}\(A(u)\n w^n\)+Q_k(\n v^n,\n w^n)+T_k(v^n,w^n)+Z_k\)&\text{(x,t)}\in\Om\times \Real^+,\\[1ex]
 (v^n(x,0),w^n(x,0)=(P_{n}(V_k),P_{n}(W_k))\quad&\text{x}\in\Om.
 \end{cases}
 \end{equation}
Obviously, the equation admits a solution $(v^n,w^n)$ in $H_n$, defined on $\Om\times [0,T^n)$, where $T^n$ is the maximal existence time. In fact, it is easy to see that $T^n=T^*$, since equation \eqref{comp-llgsp2} is a linear system and its coefficient is well controlled on $[0,T]$ for any $T<T^*$.

To show the $H^3$-estimates of $v^n$, we firstly choose $v^n$ and $-\De^2 v^n$ as test functions. A simple computation shows
 \begin{equation}\label{L^2-es-v^n}
 \begin{aligned}
 &\frac{1}{2}\frac{\p}{\p_t}\int_{\Om}|v^n|^2dx+\al \int_{\Om}|\n v^n|^2dx\\
 \leq& C(\al,\Phi )(\norm{u}^2_{H^3}+\norm{u}^4_{H^2}+\norm{s}^4_{H^2}+1)\int_{\Om}|v^n|^2dx+C\int_{\Om}|w^n|^2dx+\norm{F_k}^2_{L^2},
 \end{aligned}
 \end{equation}
 and
 \begin{equation}\label{H^2-es-v^n}
 \begin{aligned}
 &\frac{1}{2}\frac{\p}{\p_t}\int_{\Om}|\De v^n|^2dx+\al \int_{\Om}|\n\De v^n|^2dx\\
 =&-\int_{\Om}\<\n u\times \De v^n+\n K_k+\n L_k+\n F_k,\n \De v^n\>dx.
 \end{aligned}
 \end{equation}
 Here
 $$$$
 \begin{align*}
\left|\int_{\Om}\<\n u\times \De v^n,\n \De v^n\>dx\right|\leq& C(\ep,\al)\norm{u}^2_{H^3(\Om)}\int_{\Om}|\De v^n|^2dx+\ep \al \int_{\Om}|\n \De v^n|^2dx,\\
\left|\int_{\Om}\<\n K_k,\n \De v^n\>dx\right|\leq& C(\ep,\al,\Phi)(\norm{u}^2_{H^3(\Om)}+\norm{u}^4_{H^2(\Om)}+1)\norm{v^n}^2_{H^2(\Om)}\\
&+ \ep \al \int_{\Om}|\n \De v^n|^2dx,\\
\left|\int_{\Om}\<\n L_k ,\n \De v^n\>dx\right|
\leq& C(\ep,\al,\Phi)(\norm{s}^2_{H^3(\Om)}+\norm{u}^2_{H^2(\Om)}+1)(\norm{s}^2_{H^2(\Om)}+\norm{u}^2_{H^2(\Om)}+1)\\
& \cdot(\norm{v^n}^2_{H^1(\Om)}+\norm{w^n}^2_{H^1(\Om)})+\ep \al \int_{\Om}|\n \De v^n|^2dx.
 \end{align*}
To combining the above equalities, \eqref{H^2-es-v^n} can be rewritten as the following
  \begin{equation}\label{newH^2-es-v^n}
 \begin{aligned}
&\frac{\p}{\p_t}\int_{\Om}|\De v^n|^2dx+\al \int_{\Om}|\n\De v^n|^2dx\\
\leq &C(\al, \Phi)(\norm{u}^2_{H^3(\Om)}+\norm{s}^2_{H^3(\Om)}+\norm{u}^4_{H^2(\Om)}+1)(\norm{s}^2_{H^2(\Om)}+\norm{u}^2_{H^2(\Om)}+1)\\
&(\norm{v^n}^2_{H^2(\Om)}+\norm{w^n}^2_{H^2(\Om)})+\norm{F_{k}}^2_{H^1(\Om)}.
 \end{aligned}
 \end{equation}

Moreover, to show the $H^2$-estimates of $w^n$, we choose $w^n$ and $-\De^2 w^n$ as test functions. A simple computation shows
 \begin{equation}\label{L^2-es-w^n}
 \begin{aligned}
 &\frac{1}{2}\frac{\p}{\p_t}\int_{\Om}|w^n|^2dx+(1-\theta) \int_{\Om}|D_0\n w^n|^2dx\\
 \leq& C(\th,\norm{D_0}_{C^1},\norm{J_e}_{H^2(\Om)})\norm{v^n}_{H^1(\Om)}+C\int_{\Om}|w^n|^2dx+\norm{Z_k}^2_{L^2(\Om)},
 \end{aligned}
 \end{equation}
 and
 \begin{equation}\label{H^2-es-w^n}
 \begin{aligned}
 &\frac{1}{2}\frac{\p}{\p_t}\int_{\Om}|\De w^n|^2dx\\
 =&-\int_{\Om}\<-\n (\mbox{div}(A(u)\n w^n))+\n Q_k+\n T_k+\n Z_k,\n \De w^n\>dx.
 \end{aligned}
 \end{equation}
We can see easily that
\begin{align*}
 \int_{\Om}\<\n \(\mbox{div}(A(u)\n w^n)\),\n \De w^n\>dx\leq&C(\th ,\norm{D_0}_{C^2},c_0)(\norm{u}^2_{H^3(\Om)}+\norm{u}^4_{H^3(\Om)}+1)\norm{w^n}^2_{H^2(\Om)}\\
 & -\frac{3}{4}(1-\th)\int_{\Om}D_0|\n \De w^n|^2dx,
\end{align*}
\begin{align*}
\left|\int_{\Om}\<\n Q_k ,\n \De w^n\>dx\right|\leq &C(\ep, \th ,\norm{D_0}_{C^2},c_0,\norm{J_e}_{H^2(\Om)})\norm{v^n}^2_{H^2(\Om)}\\
&+\ep \th \int_{\Om}|\n \De w^n|^2dx,
\end{align*}
and
\begin{align*}
\left|\int_{\Om}\<\n T_k ,\n \De w^n\>dx\right|\leq& C(\ep, \th ,\norm{D_0}_{C^2},c_0,\norm{J_e}_{H^2(\Om)})(\norm{s}^2_{H^3(\Om)}+\norm{u}^2_{H^3(\Om)}\norm{u}^2_{H^3(\Om)}\\
&+\norm{u}^2_{H^3(\Om)}+1)(\norm{v^n}^2_{H^2(\Om)}+\norm{w^n}^2_{H^2(\Om)})\\
&+\ep\th \int_{\Om}|\n \De w^n|^2dx.
\end{align*}
In view of the above three inequalities, from \eqref{H^2-es-w^n} we deduce the following inequality
\begin{equation}\label{newH^2-es-w^n}
\begin{aligned}
&\frac{\p}{\p t}\int_{\Om}|\De w^n|^2dx+(1-\th)\int_{\Om}D_0|\n \De w^n|^2dx\\
\leq& C(\th, c_0,\norm{D_0}_{C^2},\norm{J_e}_{H^2(\Om)})(\norm{s}^4_{H^3(\Om)}+\norm{u}^2_{H^3(\Om)}+1)(\norm{v^n}^2_{H^2(\Om)}\\
&+\norm{w^n}^2_{H^2(\Om)})+\norm{Z_{k}}^2_{H^1(\Om)}.
\end{aligned}
\end{equation}
Hence, by combining inequalities \eqref{L^2-es-v^n}, \eqref{newH^2-es-v^n}, \eqref{L^2-es-w^n} and \eqref{newH^2-es-w^n} we get the following
\begin{equation}\label{H^2-es-v^nw^n}
\begin{aligned}
&\frac{\p}{\p t}(\norm{v^n}^2_{H^2(\Om)}+\norm{w^n}^2_{H^2(\Om)})+\al\int_{\Om}|\n \De v^n|^2dx+(1-\th)\int_{\Om}D_0|\n \De w^n|^2dx\\
\leq& C(\th, \al, \Phi, c_0,\norm{D_0}_{C^2},\norm{J_e}_{H^2(\Om)})(\norm{v^n}^2_{H^2(\Om)}+\norm{w^n}^2_{H^2(\Om)})p(t) + q(t).
\end{aligned}
\end{equation}
Here
$$q(t)=\norm{F_k}^2_{H^1(\Om)}+\norm{Z_k}^2_{H^1(\Om)}$$
and
$$p(t)=(\norm{s}^4_{H^3(\Om)}+\norm{u}^4_{H^3(\Om)}+1)(\norm{s}^2_{H^3(\Om)}+\norm{u}^2_{H^3(\Om)}+1).$$
Moreover, Proposition \ref{es-F_kZ_k} tells us
$$q(t)\leq C(T) \quad\mbox{and}\quad p(t)=(\norm{s}^4_{H^3(\Om)}+\norm{u}^4_{H^3(\Om)}+1)(\norm{s}^2_{H^3(\Om)}+\norm{u}^2_{H^3(\Om)}+1)\leq C(T).$$
Then, by the Gronwall inequality we infer from (\ref{H^2-es-v^nw^n})
\begin{equation}\label{newH^2-es-v^nw^n}
\begin{aligned}
&\sup_{0< t\leq T}(\norm{v^n}^2_{H^2(\Om)}+\norm{w^n}^2_{H^2(\Om)})+\al\int_{0}^{T}\int_{\Om}|\n \De v^n|^2dxdt\\
&+(1-\th)\int_{0}^{T}\int_{\Om}|\n \De w^n|^2dxdt\leq C(T, \norm{V_k}_{H^2(\Om)}+\norm{W_k}_{H^2(\Om)}),
\end{aligned}
\end{equation}
where $0<T<T^*$. Here we have used the fact
$$\norm{V_k}_{H^2(\Om)}+\norm{W_k}_{H^2(\Om)}\leq C(\al, \th, \norm{D_0}_{C^{2k+1}(\Om)}, \norm{u_0}_{H^{2k+2}(\Om)}+\norm{s_0}_{H^{2k+2}(\Om)}).$$

Therefore, the compactness lemma (Lemma \ref{A-S}) claims that the limiting map
$$(v,w)\in L^{\infty}([0,T],H^2(\Om,\Real^3))\cap L^2([0,T],H^3(\Om,\Real^3))$$
of sequence $\{(v^n, w^n)\}$ as $n\to \infty$ is just a solution to \eqref{comp-llgsp2}, which satisfies the same estimates as \eqref{H^2-es-v^nw^n} and \eqref{newH^2-es-v^nw^n} with replacing $(v^n,w^n)$ by $(v,w)$.

\subsubsection{\textbf{Uniqueness of Solution to \eqref{comp-llgsp2}}}
In this part, we will show the uniqueness of the solutions to \eqref{comp-llgsp2} in $L^{\infty}([0,T],H^1(\Om,\Real^3))\cap L^2([0,T],H^2(\Om,\Real^3))$.
\begin{prop}
In the space $L^{\infty}([0,T],H^1(\Om,\Real^3))\cap L^2([0,T],H^2(\Om,\Real^3))$, there exists a unique solution to \eqref{comp-llgsp2}.
\end{prop}
\begin{proof}
Let $(v,w)$ and $(\tilde{v},\tilde{w})$ are two solutions. We denote $(\bar{v},\bar{w})=(v-\tilde{v},w-\tilde{w})$, which satisfies the following equation
\begin{equation*}
\begin{cases}
\p_t \bar{v}=\al \De \bar{v}-u\times \De \bar{v}+K_k(\n \bar{v},\n \bar{w})+L_k(\bar{v},\bar{w}),&\text{(x,t)}\in\Om\times (0,T^*),\\[1ex]
\p_t \bar{w}=-\mbox{div}\(A(u)\n \bar{w}\)+Q_k(\n \bar{v},\n \bar{w})+T_k(\bar{v},\bar{w}),&\text{(x,t)}\in\Om\times (0,T^*),\\[1ex]
\frac{\p \bar{v}}{\p \nu}=0,\quad\quad&\text{(x,t)}\in\p \Om\times [0,T^*),\\[1ex]
\frac{\p \bar{w}}{\p \nu}=0,\quad\quad&\text{(x,t)}\in\p \Om\times [0,T^*),\\[1ex]
(\bar{v}(x,0),\bar{w}(x,0))=(0,0),\quad&x\in\Om.
\end{cases}
\end{equation*}
Thus, for any fixed $0<T<T^*$, by choosing $(\bar{v},\bar{w})$ as text function to the above equation, there hold true
\begin{align*}
&\frac{\p}{\p t}\int_{\Om}|\bar{v}|^2dx+\al \int_{\Om}|\n \bar{v}|^2dx\\
\leq &C(\al, \Phi)(\norm{u}^2_{H^3(\Om)}+\norm{s}^2_{H^2(\Om)}+1)\int_{\Om}|\bar{v}|^2dx +C\int_{\Om}|\bar{w}|^2dx,
\end{align*}
and
\begin{align*}
&\frac{\p}{\p t}\int_{\Om}|\bar{w}|^2dx+(1-\th) \int_{\Om}|\n \bar{w}|^2dx\\
\leq& C(\th,\norm{D_0}_{C^1(\Om)} \norm{J_e}^2_{H^2(\Om)})\norm{s}^2_{H^3(\Om)}\left(\int_{\Om}(|\bar{v}|^2+|\bar{w}|^2)dx\right).
\end{align*}
It follows
\begin{align*}
&\frac{\p}{\p t}\int_{\Om}(|\bar{v}|^2+|\bar{w}|^2)dx\\
\leq&C(\al,,\th,\Phi,\norm{D_0}_{C^1(\Om)},\norm{J_e}_{H^2(\Om)})(\norm{u}^2_{H^3(\Om)}+\norm{s}^2_{H^3(\Om)}+1)\left(\int_{\Om}(|\bar{v}|^2+|\bar{w}|^2)dx\right).
\end{align*}
Thus, the Gronwall inequality tells us that, for any $0<t<T^*$,
$$(\bar{v}(x,t),\bar{w}(x,t))=((\bar{v}(x,0),\bar{w}(x,0)))=(0,0).$$
\end{proof}

Immediately, $(u_k,s_k)\equiv(v,w)$ follows the uniqueness result, since $(u_k,s_k)$ is also a solution to \eqref{comp-llgsp2} in the space $L^{\infty}([0,T],H^1(\Om,\Real^3))\cap L^2([0,T],H^2(\Om,\Real^3))$.\medskip

\subsubsection{\textbf{The proof of property $\P(k+1)$}}\
Next, we are in the position to show the item $(1)$ of property $\P(k+1)$ holds true if $(u_0,s_0)\in H^{2k+2}(\Om,\Real^3)$ and the solution to \eqref{llgsp2} satisfies the property $\P(k)$.
\begin{prop}\label{ve-reg-us}
If $(u_0,s_0)\in H^{2k+2}(\Om,\Real^3)$ and the property $\P(k)$ holds, then we have
\begin{equation}\label{reg-us}
(\p^i_t u,\p^i_t s)\in L^\infty([0,T], H^{2(k+1)-2i}(\Om,\Real^3))\cap L^2([0,T], H^{2(k+1)-2i+1}(\Om,\Real^3)),
\end{equation}
where $0<T<T^*$ and $i\in \{0,\dots,k+1\}$.
\end{prop}
\begin{proof}
We use mathematical induction on $k+1-l$. When $l=0$, we have
\begin{equation*}
\begin{cases}
u_{k+1}=\al \De u_k-u\times \De u_k+K_k(\n u_k,\n s_k)+L_k(s_k,u_k)+F_k(u,s),\\[1ex]
s_{k+1}=-\mbox{div}\(A(u)\n s_k\)+Q_k(\n u_k,\n s_k)+T_k(u_k,s_k)+Z_k(u,s).
\end{cases}
\end{equation*}
A direct computation shows
$$(u_{k+1},s_{k+1})\in L^\infty([0,T],L^2(\Om,\Real^3))\cap L^2([0,T],H^1(\Om,\Real^3)),$$
where we have used Proposition  \ref{es-F_kZ_k} and estimate \eqref{newH^2-es-v^nw^n}.

Thus, we have shown the result holds when $l=0$ and $1$. Now, we assume $l=i\geq 1$, the result has been proved, then, we need to establish it for $l=i+1$, where $i\leq k-1$.
Since $u_{k+1-i}=\p_t u_{k-i}$, it follows
\begin{equation*}
\begin{cases}
	\al \De u_{k-i}-u\times \De u_{k-i}=u_{k+1-i}-K_{k-i}(\n u_{k-i},\n s_{k-i})-L_{k-i}(s_{k-i},u_{k-i})+F_{k-i}(u,s),\\[1ex]
	\mbox{div}\(A(u)\n s_{k-i}\)=-s_{k+1-i}+Q_{k-i}(\n u_{k-i},\n s_{k-i})+T_{k-i}(u_{k-i},s_{k-i})+Z_{k-i}(u,s).
\end{cases}
\end{equation*}
Next, we consider the first equation in order to obtain the estimate of $u_{k-i}$. By utilizing the properties $\P(k)$ and Proposition \ref{es-F_kZ_k}, we have
\begin{itemize}
	\item $u_{k-i}\in L^\infty([0,T],H^{2i+1}(\Om,\Real^3))\cap L^2([0,T],H^{2i+2}(\Om,\Real^3))$,
	\item $u\in L^\infty([0,T],H^{2k+1}(\Om,\Real^3))\cap L^2([0,T],H^{2k+2})$,
	\item $F_{k-i}\in L^\infty([0,T],H^{2i+1})\cap L^2([0,T],H^{2i+2}(\Om,\Real^3))$,
	\item and by the assumption of induction,
	\[u_{k+1-i}\in L^\infty([0,T],H^{2i}(\Om,\Real^3))\cap L^2([0,T],H^{2i+1}(\Om,\Real^3)).\]
\end{itemize}

For the term $K_{k-i}$, since $\n u_{k-i}$ and $\n s_{k-i}$ are in $L^{\infty}([0,T],H^{2i}(\Om,\Real^3))\cap L^{2}([0,T],H^{2i+1}(\Om,\Real^3))$,  $u$ is in $L^{\infty}([0,T],H^{2k+1}(\Om,\Real^3))$, then
$$K_{k-i}\in L^{\infty}([0,T],H^{2i}(\Om,\Real^3))\cap L^{2}([0,T],H^{2i+1}(\Om,\Real^3)),$$
where Lemma \ref{algebra2} has been used.

Since $\De u\in L^{\infty}([0,T],H^{2i+1}(\Om,\Real^3))\cap L^{2}([0,T],H^{2i+2}(\Om,\Real^3))$ by $2k-1\geq 2i+1$, it follows
 $$u_{k-i}\times \De u\in L^{\infty}([0,T],H^{2i}(\Om,\Real^3))\cap L^{2}([0,T],H^{2i+1}(\Om,\Real^3)).$$

By an almost the same argument as that for $K_{k-i}$, we know
 $$L_{k-i}\in L^{\infty}([0,T],H^{2i}(\Om,\Real^3))\cap L^{2}([0,T],H^{2i+1}(\Om,\Real^3)),$$
since $h_d(u_{k-i})$ and $s_{k-i}$ are in $L^{\infty}([0,T],H^{2i+1}(\Om,\Real^3))\cap L^{2}([0,T],H^{2i+2}(\Om,\Real^3))$.

Now, it is not difficult to see that the above estimates imply
$$\al\De u_{k-i}-u\times \De u_{k-i}\in L^{\infty}([0,T],H^{2i}(\Om,\Real^3))\cap L^{2}([0,T],H^{2i+1}(\Om,\Real^3)).$$
Furthermore, the $L^p$-estimate of elliptic equation shows
$$u_{k-i}\in L^{\infty}([0,T],H^{2i+2}(\Om,\Real^3))\cap L^{2}([0,T],H^{2i+3}(\Om,\Real^3)).$$

On the other hand, in order to show the same estimate of $s_{k-i}$ we need to take almost the same argument as that for $u_{k-i}$ except for we need to control the following term $u_{k-i}\ast\n^2 s\ast u$. Since $u_{k-i}$, $\n^2 s$ and $u$ are all in $L^\infty([0,T],H^{2i+1}(\Om,\Real^3))$, it follows
$$u_{k-i}\ast\n^2 s\ast u\in L^{\infty}([0,T],H^{2i}(\Om,\Real^3))\cap L^{2}([0,T],H^{2i+1}(\Om,\Real^3)).$$
Note that, to control the term $Q_k$, here we need $J_e\in L^\infty([0,T], H^{2i}(\Om,\Real^3))\cap L^{2}([0,T],H^{2i+1}(\Om,\Real^3))$.

Thus, we have
$$\mbox{div}(A(u)\n s_{k-i})\in L^{\infty}([0,T],H^{2i}(\Om,\Real^3))\cap L^{2}([0,T],H^{2i+1}(\Om,\Real^3)),$$
it follows
$$s_{k-i}\in L^{\infty}([0,T],H^{2i+2}(\Om,\Real^3))\cap L^{2}([0,T],H^{2i+3}(\Om,\Real^3)).$$
Therefore, we finish the induction argument. In particular, when $l=k$, we have
$$(\p_t u,\p_t s)\in L^{\infty}([0,T],H^{2k}(\Om,\Real^3))\cap L^{2}([0,T],H^{2k+1}(\Om,\Real^3)).$$

Finally, we need to show the result when $l=k+1$. Since $(\p_t u,\p_t s)$ satisfies the following equations
\begin{equation*}
\begin{cases}
\De u=-|\n u|^2u+u\times (u\times(\tilde{h}+s))-\frac{1}{1+\al^2}\(\al \p_t u+u\times \p_t u\),\quad\quad&\text{(x,t)}\in\Om\times \Real^+,\\[1ex]
\mbox{div}(A(u)\cdot \n s)=-\p_t s+\mbox{div}(u\otimes J_e)+D_0(x)\cdot s+D_0(x)\cdot s\times u, \quad\quad&\text{(x,t)}\in\Om\times \Real^+.
\end{cases}
\end{equation*}
In view of the fact
$$(u,s)\in L^\infty([0,T],H^{2k+1}(\Om,\Real^3))\cap L^2([0,T],H^{2k+2}(\Om,\Real^3)),$$
we take a bootstrap argument to show
$$(\De u, \De s)\in L^\infty([0,T],H^{2k}(\Om,\Real^3))\cap L^2([0,T],H^{2k+1}(\Om,\Real^3)).$$
Hence, it implies
$$(u,s)\in L^\infty([0,T],H^{2k+2}(\Om,\Real^3))\cap L^2([0,T],H^{2k+3}(\Om,\Real^3)),$$
by $L^p$-estimates.
\end{proof}

Now, assume that $(u_0,s_0)\in H^{2(k+1)+1}(\Om,\Real^3)$. We want to show the item (2) in the property $\P(k+1)$. Here, we will only give the sketches of proof to this property, since the proof goes almost the same as that in Section \ref{ss: 1-reg} and the proof of Proposition \ref{ve-reg-us}. First of all, we prove the following result, which is analogous to Theorem \ref{mth2}.

\begin{prop}\label{v-reg-us}
 If $(u_0,s_0)\in H^{2(k+1)+1}(\Om,\Real^3)$ and the property $\P(k)$ holds, then
\begin{equation}
 (u_{k+1},s_{k+1})\in L^\infty([0,T],H^1(\Om,\Real^3))\cap L^2([0,T], H^2(\Om,\Real^3)),
 \end{equation}
 for any $0<T<T^*$.
 \end{prop}
\begin{proof}
By the Galerkin approximation equation \eqref{app-comp-llgsp1} associated to \eqref{comp-llgsp2}, $(v^n_t,w^n_t)=(\p_t v^n, \p_t w^n)$ satisfies the following equation
\begin{equation}\label{app-comp-llgsp2}
\begin{cases}
\p_t v^n_t=\al\De v^n_t+ P_n\p_t\(-u\times \De v^n+K_k(\n v^n,\n w^n)+L_k(v^n,w^n)+F_k\),&\text{(x,t)}\in\Om\times [0,T^*),\\[1ex]
\p_t w^n_t=P_n\p_t\(-\mbox{div}\(A(u)\n w^n\)+Q_k(\n v^n,\n w^n)+T_k(v^n,w^n)+Z_k\),&\text{(x,t)}\in\Om\times [0,T^*).
\end{cases}
\end{equation}
By the assumption of $\P(k+1)$ and the previous induction arguments, we can combine the estimates in Proposition \ref{es-F_kZ_k}, Proposition \ref{ve-reg-us} and estimate \eqref{newH^2-es-v^nw^n} to obtain
\begin{itemize}
	\item $(v^n,w^n)\in L^\infty([0,T], H^2(\Om,\Real^3))\cap L^2([0,T],H^3(\Om,\Real^3))$,
	\item $(\p^i_tu,\p^i_ts)\in L^\infty([0,T], H^{2k+2-2i}(\Om,\Real^3))\cap L^2([0,T],H^{2k+3-2i}(\Om,\Real^3))$, where $i\in \{0,\dots,k+1\}$,
	\item $(F_k,Z_k)\in L^\infty([0,T], H^{1}(\Om,\Real^3))\cap L^2([0,T],H^{2}(\Om,\Real^3))$,
	\item $(\p_tF_k,\p_tZ_k)\in L^2([0,T],L^2(\Om,\Real^3))$.
\end{itemize}

In the following context, we aim at proving $$(v^n_t, w^n_t)\in L^\infty([0,T], H^1(\Om,\Real^3))\cap L^2([0,T],H^2(\Om,\Real^3)).$$
From equation \eqref{app-comp-llgsp1} and the estimate of $(v^n,w^n)$, we can get easily
$$\norm{(v^n_t,w^n_t)}_{L^\infty([0,T],L^2(\Om))}+\norm{(v^n_t,w^n_t)}_{L^2([0,T],H^1(\Om))}\leq C(T).$$
By choosing $(-\De v^n_t,-\De w^n_t)$ as a test function, we can show the $H^2$-estimate as follows.
\begin{equation}\label{es-v^n_t}
\begin{aligned}
&\quad \frac{\p}{\p t}\int_{\Om}|\n v^n_t|^2dx+\al \int_{\Om}|\De v^n_t|^2dx\\
&\leq C(\al)\(\int_{\Om}|K_k(\n v^n_t, \n w^n_t)|^2dx+\int_{\Om}|L_k(v^n_t, w^n_t)|^2dx+\int_{\Om}|\tilde{F}_k|^2dx\)\\
&=C(\al)(I_1+II_1+III_1).
\end{aligned}
\end{equation}
\begin{equation}\label{es-w^n_t}
\begin{aligned}
&\frac{\p}{\p t}\int_{\Om}|\n w^n_t|^2dx+(1-\th) \int_{\Om}D_0|\De w^n_t|^2dx\\
\leq& C(c_0,\th, \norm{D_0}_{C^1})\(\int_{\Om}|Q_k(\n v^n_t, \n w^n_t)|^2dx+\int_{\Om}|T_k(v^n_t, w^n_t)|^2dx\)\\
&+C(c_0,\th, \norm{D_0}_{C^1})\(\int_{\Om}|\tilde{Z}_k|^2dx+\int_{\Om}|\n w^n_t|^2dx+\int_{\Om}|\n w^n_t|^2|\n u|^2dx\)\\
=&C(c_0,\th, \norm{D_0}_{C^1})(I_2+II_2+III_2+IV_2+V_2)
\end{aligned}
\end{equation}
where
\begin{align*}
\tilde{F_k}=&-u_t\times\De v^n+\p_t K_k(\n v^n, \n w^n)+\p_tL_k(v^n,w^n)+\p_t F_{k},\\
\tilde{Z_k}=&-\mbox{div}\(\p_tA(u)\n w^n\)+\p_t Q_k(\n v^n, \n w^n)+\p_tT_k(v^n,w^n)+\p_t Z_{k}.
\end{align*}

By a direct computation, we get the below estimates.
\begin{equation*}
\begin{aligned}
I_1\leq &C\int_{\Om}|\n v^n_t|^2|\n u|^2\,dx\leq C\norm{u}^2_{H^3}\int_{\Om}|\n v^n_t|^2dx,\\
II_1\leq& 2\int_{\Om}|v^n_t|^2|\n u|^4\,dx+\int_{\Om}|v^n_t|^2(|\tilde{h}|^2+|s|^2)dx\\
&+ C\int_{\Om}(|h_d(v^n_t)|^2+|v_t|^2+|w^n_t|^2)dx+\int_{\Om}|v^n|^2(|h(u)|^2+|s|^2)dx\\
\leq &C(\norm{u}^4_{H^3(\Om)}+\norm{s}^2_{H^2(\Om)}+1)\int_{\Om}|v^n_t|^2+C\int_{\Om}|w^n_t|^2dx+C\norm{u}^2_{H^3}\norm{v^n_t}_{H^1}\\
\leq& C(T)+C(T)\int_{\Om}|\n v^n_t|^2\,dx,\\
III_1\leq& \int_{\Om}|\p_t K_k(\n v^n, \n w^n)|^2\,dx+\int_{\Om}|\p_tL_k(v^n,w^n)|^2\,dx+\int_{\Om}|\p_t F_{k}|^2dx\\
&+\int_{\Om}|u_t|^2|\De v^n|^2\,dx\leq \norm{v^n}^2_{H^2(\Om)}\norm{u_t}^2_{H^2(\Om)}+\norm{u}^2_{H^3(\Om)}\norm{\p_t u}^2_{H^2(\Om)}\int_{\Om}|\n v^n|^2\,dx\\
&+\int_{\Om}|\p_tL_k(v^n,w^n)|^2\,dx+\int_{\Om}|\p_t F_{k}|^2dx+\norm{u_t}^2_{H^2(\Om)}\int_{\Om}|\De v^n|^2\,dx\\
\leq &C(T)+\int_{\Om}|\p_t F_{k}|^2dx.
\end{aligned}
\end{equation*}
Here,
\begin{equation*}
\begin{aligned}
\int_{\Om}|\p_tL_k(v^n,w^n)|^2\,dx\leq &C\norm{u}^2_{H^2(\Om)}\norm{v^n}^2_{H^2(\Om)}\norm{u_t}_{H^2(\Om)}+C\norm{v^n}_{H^2(\Om)}\norm{u_t}^2_{H^2(\Om)}(1+\norm{s}^2_{L^2(\Om)})\\
&+C\norm{v^n}^2_{H^2(\Om)}(\norm{u_t}^2_{L^2(\Om)}+\norm{s_t}^2_{L^2(\Om)})\\
&+C\norm{u_t}^2_{L^2(\Om)}(\norm{v^n}^2_{H^2(\Om)}+\norm{w^n}^2_{H^2(\Om)})+\norm{v^n}^2_{H^2(\Om)}\int_{\Om}|\De u_t|^2\,dx\\
\leq & C(T).
\end{aligned}
\end{equation*}

Thus, there holds
\begin{equation*}
\begin{aligned}
&\quad \frac{\p}{\p_t}\int_{\Om}|\n v^n_t|^2dx+\al \int_{\Om}|\De v^n_t|^2dx\\
&\leq C(T)+C(T)\int_{\Om}|\p_t F_k|^2dx+C(T)\int_{\Om}|\n v^n_t|^2dx.
\end{aligned}
\end{equation*}
By Gronwall inequality, it follows
$$\sup_{0< t\leq T}\norm{v^n_t}^2_{H^1(\Om)}+\int_{0}^{T}\norm{v^n_t}^2_{H^2(\Om)}\,dt\leq C(T, \norm{v^n_t|_{t=0}}^2_{H^1(\Om)}).$$
Here we have used the fact
$$\int_{\Om}|\p_t F_k|^2dx+\int_{\Om}|\n v^n_t|^2dx\in L^1([0,T]).$$

Now, we turn to showing the $H^2$-estimate of $w^n_t$. We need to control the terms in the right hand side of inequality \eqref{es-w^n_t} as follows.
\begin{equation*}
\begin{aligned}
I_2\leq &C\norm{J_e}^2_{H^2(\Om)}\norm{v^n_t}^2_{H^{1}(\Om)}\leq C(T,\norm{v^n_t|_{t=0}}^2_{H^1}),\\
II_2\leq& C(\norm{D_0}_{C^1},c_0,\th)\norm{s}^2_{H^3(\Om)}\norm{ v^n_t}^2_{H^{1}(\Om)}\\
&+C(\norm{D_0}_{C^1},c_0,\th)\left(\norm{s}^2_{H^3(\Om)}\norm{u}^2_{H^3(\Om)}\norm{v^n_t}^2_{L^{2}(\Om)}+\int_{\Om}|w^n_t|^2dx+\norm{s}^2_{H^2}\int_{\Om}|v^n_t|^2dx\right)\\
\leq &C(T),\\
III_2\leq &\int_{\Om}|\p_t Q_k(\n v^n,\n w^n)|^2\,dx+\int_{\Om}|\p_t T_k(v^n,w^n)|^2\,dx+\int_{\Om}|\p_t Z_k|^2\,dx\\
\leq &C\norm{\p_t J_e}^2_{H^1(\Om)}\norm{v^n}^2_{H^2(\Om)}+\int_{\Om}|\p_t Z_k|^2\,dx+C(T).
\end{aligned}
\end{equation*}
Here, we have used the below estimates
\begin{equation*}
\begin{aligned}
\int_{\Om}|\p_t T_k(v^n,w^n)|^2\,dx
\leq &C(\norm{D_0}_{C^1})\left(\norm{v^n}^2_{H^2(\Om)}\int_{\Om}|\n s_t|^2\,dx+\norm{u_t}^2_{H^2(\Om)}\norm{v^n}^2_{H^2}\int_{\Om}|\n s|^2\,dx\right)\\
&+C(\norm{D_0}_{C^1})\left(\norm{s}^2_{H^3(\Om)}\int_{\Om}|\p_t v^n|^2\,dx+\norm{v^n}^2_{H^2(\Om)}\int_{\Om}|\n s|^2\,dx\right)\\
&+C(\norm{D_0}_{C^1})\left(\norm{u_t}^2_{H^2(\Om)}\int_{\Om}|w^n|^2\,dx+\norm{s_t}^2_{H^2(\Om)}\int_{\Om}|v^n|^2\,dx\right)\\
\leq &C(T),
\end{aligned}
\end{equation*}
and
\begin{equation*}
\begin{aligned}
&\int_{\Om}|\mbox{div}(D_0\th u_t\otimes u\cdot \n w^n)|^2\,dx\\
\leq& C(\norm{D_0}_{C^1})\left(\norm{u_t}^2_{H^2(\Om)}\norm{w^n}^2_{H^2(\Om)}+\norm{u_t}^2_{H^2(\Om)}\norm{u}^2_{H^3(\Om)}\int_{\Om}|\n w^n|^2\,dx\right)\leq C(T).
\end{aligned}
\end{equation*}
Therefore, we get
\begin{equation*}
\begin{aligned}
&\frac{\p}{\p_t}\int_{\Om}|\n w^n_t|^2dx+(1-\th) \int_{\Om}D_0|\De w^n_t|^2dx\\
\leq& C(T)+C(T)\int_{\Om}|\p_t Z_k|^2\,dx+C(T)\int_{\Om}|\n w^n_t|^2dx.
\end{aligned}
\end{equation*}
Hence, the Gronwall inequality implies
$$\sup_{0< t\leq T}\norm{w^n_t}^2_{H^1(\Om)}+\int_{0}^{T}\norm{w^n_t}^2_{H^2(\Om)}\,dt\leq C(T, \norm{v^n_t|_{t=0}}^2_{H^1(\Om)},\norm{w^n_t|_{t=0}}^2_{H^1(\Om)}).$$

Finally, we need to show the assumption of initial date $(u_0,s_0)$ can guarantee the boundness of $\norm{\n v^n_t}_{L^2}+\norm{\n w^n_t}_{L^2}$, hence the proof is completed. Since the initial date of \eqref{app-comp-llgsp2} satisfies
\begin{equation*}
\begin{cases}
v^n_t(x,0)=\al \De P_n (V_k)+P_n(-u_0\times P_n(V_k)+K_k|_{t=0}(\n P_n(V_k),\n P_n(W_k))\\
\quad \quad\quad\quad\quad+L_k|_{t=0}(P_n(V_k),P_n(W_k))+F_k|_{t=0}),\\[2ex]
w^n_t(x,0)=P_n(-\mbox{div}(A(u_0)\n P_n(W_k))+Q_k|_{t=0}(\n P_n(V_k),\n P_n(W_k))\\
\quad\quad\quad\quad\quad+T_k|_{t=0}(P_n(V_k),P_n(W_k))+Z_k|_{t=0}),
\end{cases}
\end{equation*}
a simple calculation shows
$$\int_{\Om}|\n v^n_t|_{t=0}|^2\,dx\leq C(\norm{P_n(V_k)}^2_{H^3(\Om)}+\norm{P_n(W_k)}^2_{H^3(\Om)}+1),$$
and
$$\int_{\Om}|\n w^n_t|_{t=0}|^2\,dx\leq C(\norm{P_n(V_k)}^2_{H^3(\Om)}+\norm{P_n(W_k)}^2_{H^3(\Om)}+1).$$
Thus, there holds
\begin{equation*}
\begin{aligned}
\int_{\Om}|\n v^n_t|_{t=0}|^2\,dx+\int_{\Om}|\n w^n_t|_{t=0}|^2\,dx&\leq C(\norm{V_k}^2_{H^3(\Om)}+\norm{W_k}^2_{H^3(\Om)}+1)\\
&\leq C(\norm{u_0}_{H^{2k+3}(\Om)},\norm{s_0}_{H^{2k+3}(\Om)}).
\end{aligned}
\end{equation*}
Here, we have applied the formula of $V_k$ and Lemma \ref{es-P_n} about estimate of $P_n$.
\end{proof}

By summarizing the estimates in Propostion \ref{ve-reg-us} and Propostion \ref{v-reg-us}, and taking an almost the same argument as in the proof of Proposition \ref{ve-reg-us}, we can prove
\begin{equation}\label{vvreg-us}
(\p^i_t u,\p^i_t s)\in L^\infty([0,T], H^{2(k+1)-2i+1}(\Om,\Real^3))\cap L^2([0,T], H^{2(k+1)-2i+2}(\Om,\Real^3)),
\end{equation}
where $0<T<T^*$ and $i\in \{0,\dots,k+1\}$. Hence, the property $\P(k+1)$ is established.

\medskip\medskip
\noindent {\it\textbf{Acknowledgements}}: The authors are supported partially by NSFC grant (No.11731001). The author Y. Wang is supported partially by NSFC grant (No.11971400) and Guangdong Basic and Applied Basic Research Foundation Grant (No. 2020A1515011019).
	
\end{document}